\documentclass[11pt,a4paper,reqno]{amsart}
\usepackage{amsfonts,amsthm, amscd, epsfig, amsmath, amssymb,enumerate}

\title[Zero range processes   on  the supercritical percolation cluster]{Hydrodynamic limit of zero
range  processes among random conductances on the supercritical
percolation cluster}

\author{Alessandra Faggionato}
\address{Alessandra Faggionato. Dipartimento di Matematica ``G. Castelnuovo", Universit\`a   ``La
  Sapienza''. P.le Aldo Moro  2, 00185  Roma, Italy. e--mail:
  faggiona@mat.uniroma1.it}

\numberwithin{equation}{section}

\DeclareMathSymbol{\leqslant}{\mathalpha}{AMSa}{"36} % nicer `smaller or equal'
\DeclareMathSymbol{\geqslant}{\mathalpha}{AMSa}{"3E} % nicer `larger or equal'
\DeclareMathSymbol{\eset}{\mathalpha}{AMSb}{"3F}     % nicer `emptyset'
\renewcommand{\leq}{\;\leqslant\;}                   % redef. of < or =
\renewcommand{\geq}{\;\geqslant\;}                   % redef. of > or =
\newcommand{\be}{\begin{equation}}

\def\1{\ifmmode {1\hskip -3pt \rm{I}} \else {\hbox {$1\hskip -3pt \rm{I}$}}\fi}

\newtheorem{Th}{Theorem}[section]
\newtheorem{Le}[Th]{Lemma}

\newtheorem{Rem}{Remark}

\newcommand{\cB}{\ensuremath{\mathcal B}}
\newcommand{\cC}{\ensuremath{\mathcal C}}
\newcommand{\cD}{\ensuremath{\mathcal D}}
\newcommand{\cE}{\ensuremath{\mathcal E}}
\newcommand{\cF}{\ensuremath{\mathcal F}}
\newcommand{\cG}{\ensuremath{\mathcal G}}

\newcommand{\cI}{\ensuremath{\mathcal I}}
\newcommand{\cJ}{\ensuremath{\mathcal J}}

\newcommand{\cL}{\ensuremath{\mathcal L}}
\newcommand{\cM}{\ensuremath{\mathcal M}}
\newcommand{\cN}{\ensuremath{\mathcal N}}

\newcommand{\cS}{\ensuremath{\mathcal S}}

\newcommand{\cZ}{\ensuremath{\mathcal Z}}

\newcommand{\bbE}{{\ensuremath{\mathbb E}} }

\newcommand{\bbI}{{\ensuremath{\mathbb I}} }

\newcommand{\bbL}{{\ensuremath{\mathbb L}} }

\newcommand{\bbN}{{\ensuremath{\mathbb N}} }

\newcommand{\bbP}{{\ensuremath{\mathbb P}} }
\newcommand{\bbQ}{{\ensuremath{\mathbb Q}} }
\newcommand{\bbR}{{\ensuremath{\mathbb R}} }

\newcommand{\bbX}{{\ensuremath{\mathbb X}} }

\newcommand{\bbZ}{{\ensuremath{\mathbb Z}} }

\let\a=\alpha   \let\c=\chi \let\d=\delta  \let\e=\varepsilon
 \let\g=\gamma     \let\k=\kappa  \let\l=\lambda
      \let\o=\omega    \let\p=\pi  
  \let\s=\sigma \let\t=\tau   
  \let\z=\zeta
\let\D=\Delta   \let\G=\Gamma  \let\L=\Lambda 
\let\O=\Omega      

\def\\{\hfill\break}

\def\tthsp{\kern .083333 em}

\def\?{\mskip -10mu}
\def\indbox#1{\hbox to \parindent{\hfil\ #1\hfil} }

\newfam\msafam
\newfam\msbfam
\newfam\eufmfam
\def\hexnumber#1{%
  \ifcase#1 0\or 1\or 2\or 3\or 4\or 5\or 6\or 7\or 8\or
  9\or A\or B\or C\or D\or E\or F\fi}
\font\tenmsa=msam10 \font\sevenmsa=msam7 \font\fivemsa=msam5
\textfont\msafam=\tenmsa \scriptfont\msafam=\sevenmsa
\scriptscriptfont\msafam=\fivemsa
\edef\msafamhexnumber{\hexnumber\msafam}%
\mathchardef\restriction"1\msafamhexnumber16 \mathchardef\ssim"0218
\mathchardef\square"0\msafamhexnumber03
\mathchardef\eqd"3\msafamhexnumber2C
\def\QED{\ifhmode\unskip\nobreak\fi\quad
  \ifmmode\square\else$\square$\fi}
\font\tenmsb=msbm10 \font\sevenmsb=msbm7 \font\fivemsb=msbm5
\textfont\msbfam=\tenmsb \scriptfont\msbfam=\sevenmsb
\scriptscriptfont\msbfam=\fivemsb

\font\teneufm=eufm10 \font\seveneufm=eufm7 \font\fiveeufm=eufm5
\textfont\eufmfam=\teneufm \scriptfont\eufmfam=\seveneufm
\scriptscriptfont\eufmfam=\fiveeufm

\def\({\left(}
\def\){\right)}
\let\neper=e
\let\ii=i

\let\<=\langle
\let\>=\rangle

\outer\def\nproclaim#1 [#2]#3. #4\par{\medbreak \noindent
   \talato(#2){\bf #1 \Thm[#2]#3.\enspace }%
   {\sl #4\par }\ifdim \lastskip <\medskipamount
   \removelastskip \penalty 55\medskip \fi}
\def\thmm[#1]{#1}
\def\teo[#1]{#1}
\def\sttilde#1{%
\dimen2=\fontdimen5\textfont0 \setbox0=\hbox{$\mathchar"7E$}
\setbox1=\hbox{$\scriptstyle #1$} \dimen0=\wd0 \dimen1=\wd1
\advance\dimen1 by -\dimen0 \divide\dimen1 by 2
\vbox{\offinterlineskip%
   \moveright\dimen1 \box0 \kern - \dimen2\box1}
}
\newcommand{\bfA}{{\ensuremath{\mathbf A}} }
\newcommand{\bfB}{{\ensuremath{\mathbf B}} }
\newcommand{\bfC}{{\ensuremath{\mathbf C}} }
\newcommand{\bfa}{{\ensuremath{\mathbf a}} }
\newcommand{\bfw}{{\ensuremath{\mathbf w}} }
\begin{document}

\maketitle
%\date{\today}
%%%%%%%%%%%%%%%%%%%%%%%%%%%%%%%%%%%%%%%%%%%%%%%%%%%%
\begin{abstract}
We consider  i.i.d. random variables $\{\o (b):b \in \bbE_d \}$
parameterized by the family of  bonds in $\bbZ^d$, $d\geq 2$. The
random variable $\o(b)$ is  thought of as the conductance of bond
$b$ and  it ranges in a finite interval $[0,c_0]$.  Assuming the
probability $m$ of the event $\{\o(b)>0\}$ to be supercritical and
denoting by  $\cC(\o)$  the  unique infinite cluster associated to
the bonds with positive conductance, we study the zero range process
on $\cC(\o)$ with $\o(b)$--proportional probability rate of jumps
along bond $b$. For almost all realizations of the environment we
prove that the hydrodynamic behavior  of the zero range process is
governed  by the nonlinear heat equation $\partial_t \rho= m \nabla
\cdot ( \cD \nabla\phi(\rho/m) )$, where the matrix $\cD$ and the
function $\phi$ are $\o$--independent.  As byproduct of the above
result and the blocking effect of the finite clusters, we discuss
the bulk behavior of the zero range process on $\bbZ^d$ with
conductance field $\o$.  We do not require any ellipticity
condition.

\medskip

 \noindent \emph{Key words}: disordered system,
 bond percolation,  zero range process, hydrodynamic limit, homogenization, stochastic domination.

\smallskip
\noindent \emph{AMS 2000 subject classification}:   60K35, 60J27,
82C44.
\end{abstract}

% 60K35 Interacting random processes; statistical mechanics type models; percolation theory
%60J27 Markov chains with continuous parameter 82B20 Lattice systems (Ising, dimer, Potts, etc.) and systems on graphs
% 82C44 Dynamics of disordered systems (random Ising systems, etc.)
\section{Introduction}

Percolation  provides a simple and, at the same time,  very rich
model of disordered medium \cite{G}, \cite{Ke}. The motion of a random walker on percolation clusters has
been deeply investigated in Physics (see \cite{BH} and reference therein) and also numerous rigorous
results are now available. In the last years, for the supercritical percolation cluster
 it has been possible to prove the convergence of the diffusively rescaled random walk to the Brownian motion
for almost all realizations of the percolation   \cite{SS}, \cite{BB}, \cite{MP}, improving the
annealed  invariance principle obtained in \cite{DFGW}.
 We address here our attention to interacting random walkers,
 moving    on the supercritical Bernoulli
bond percolation cluster  with additional environmental disorder
given by random conductances (for recent results on random walks
among random conductances see \cite{BP}, \cite{M}, \cite{F} and
reference therein).

\smallskip

 Particle interactions can be of different kind. An
example is given by site exclusion, the hydrodynamic behavior of the
resulting exclusion process has been studied in \cite{F}. Another
basic example,  considered here, is the  zero range interaction:
particles lie on the sites of the infinite cluster without any
constraint, while the probability rate of the jump of a particle
from site $x$ to a neighboring site $y$ is given by $g(\eta(x) )
\o(x,y)$, where $g$ is a  suitable function on $\bbN$, $\eta(x)$ is
the number of particles at site $x$ and $\o(x,y)$ is the conductance
of the bond $\{x,y\}$. We suppose that the concuctances are i.i.d.
random variables taking value in $[0,c_0]$.

\smallskip

 The above exclusion and zero range processes are
non--gradient systems, since due to the disorder the algebraic local
current cannot be written as spatial gradient of some local
function.
Nevertheless,  thanks to the independence of  the  conductances from any bond orientation,  one can study the hydrodynamic behavior
 avoiding the heavy machinery of non--gradient particle systems
  \cite{V}, \cite{KL}[Chapter VII].
 Indeed,  in the case of exclusion processes, due to the above
 symmetry of the conductance field
 the infinitesimal variation of the occupancy number $\eta(x)$ is a
linear combination of occupancy numbers.
 This degree conservation strongly simplifies  the analysis of the limiting behavior
 of the random  empirical measure  with respect to  genuinely non--gradient
 disordered models  as in \cite{Q1}, \cite{FM}, \cite{Q2}, and can be
 reduced to an homogenization problem \cite{F}. In the case of
 zero--range processes, this degree conservation is broken.
 Nevertheless, due to the  symmetry of the conductance field,
   adapting  the method of the corrected empirical measure \cite{GJ1} to the present contest
    one  can   reduce  the proof of the hydrodynamic limit to an
 homogenization problem plus the proof of the Replacement Lemma. The resulting diffusive  hydrodynamic equation
 does not depend on the environment and keeps memory on the particle interaction.

\smallskip
 The homogenization problem has been solved in \cite{F}
also for more general random conductance fields. The core of the problem here is the proof of the
Replacement Lemma. This technical lemma  compares the particle density on
microscopic boxes with the particle density on  macroscopic boxes
and it is a  key tool in order to go from the microscopic scale  to the macroscopic one. This comparison
is usually made by moving particles along macroscopic paths by microscopic steps and then summing the local
variations  at each step. The resulting method corresponds to the so called   Moving
Particle Lemma and becomes efficient if the chosen macroscopic paths allow a spread--out
 particle flux, without any concentration in some special bond.
While for a.a. $\o$ any two points $x,y$ in
a box $\L_N$ of side $N$ centered at the origin can be connected
inside the infinite cluster  by a path $\gamma_{x,y}$ of length at
most $O(N)$  \cite{AP}, it is very hard (maybe impossible) to
exhibit such  a family of paths $\{\gamma_{x,y}\}_{x,y\in \L_N}$
with a reasonable  upper bound of the number of paths going through
a given bond $b$, uniformly in $b$. Due to this obstacle, we will
prove the Moving Particle Lemma not in its standard form, but in a
weaker form, allowing anyway to complete the proof of the
Replacement Lemma. We point out that in this step we use some
technical results of   \cite{AP}, where the chemical distance inside
the supercritical  Bernoulli bond  percolation  cluster is studied.
It is only here that we need the hypothesis of i.i.d. conductances.
Extending part of the results of \cite{AP}, one would get the
hydrodynamic limit of zero range processes among random conductances
on infinite clusters of more general conductance fields as in
\cite{F}.

\smallskip
We comment another technical problem we had to handle with.
The
discussion in \cite{GJ1} refers to the zero range process on a
finite toroidal grid
with conductances bounded above and below by some positive constants, and some steps cannot work here
 due to the presence of infinite particles. A particular care has to be devoted to the control of phenomena of  particle
 concentration  and slightly stronger homogenization results are required.

\smallskip
Finally, in the Appendix we discuss  the bulk behavior of the zero
range process on $\bbZ^d$ with i.i.d. random conductances in
$[0,c_0]$, in the case of initial distributions with slowly varying
parameter. Due to the blocking effect of the clusters with finite
size, the bulk behavior is not described by a nonlinear heat
equation.

\smallskip

We   recall that the problem of density fluctuations for the zero
range process
 on the supercritical Bernoulli bond  percolation    cluster with constant conductances
has been studied in \cite{GJ2}. Recently, the hydrodynamic limit of
other interacting particle systems on $\bbZ^d$, or fractal spaces,
with random conductances has been proved (cf. \cite{F1}, \cite{FJL},
\cite{JL}, \cite{LF}, \cite{Val}). We point out  the pioneering
paper \cite{Fr}, where J. Fritz has proved the hydrodynamic
behaviour of a one-dimensional Ginzburg-Landau model with
conservation law in the presence of random conductances.

\section{Models and results}\label{natalino}

\subsection{The environment}\label{salutare}
The   environment $\o$  modeling the disordered medium is given by
i.i.d. random variables $\bigl(\o(b)\,:\, b \in \bbE_d\bigr)$,
parameterized by the set $\bbE_d$ of non--oriented bonds in
$\bbZ^d$, $d\geq 2$.
 $\o$ and $\o(b)$ are  thought of as the  conductance field  and the
conductance at bond $b$, respectively.
 We call $\bbQ$ the law of
the field $\o$ and we assume that $ \o(b) \in [0,c_0]$ for
$\bbQ$--a.a. $\o$,  for some fixed positive constant $c_0$. Hence,
without loss of generality, we can suppose that $\bbQ$ is a
probability measure on the product space $\Omega:=[0,c_0] ^{\bbE _d}
$. Moreover, in order to simplify the notation, we write $\o(x,y)$
for the conductance $\o(b)$ if $b=\{x,y\}$. Note that
$\o(x,y)=\o(y,x)$.

\smallskip

  Consider the random graph $G (\o)= \bigl(  V (\o), E
(\o) \bigr)$ with  vertex set $V (\o)$ and bond set $E(\o)$ defined
as
\begin{align*}
& E (\o):=\bigl\{b\in \bbE_d\,:\, \o (b) >0\bigr \}\,, \\
& V (\o) :=\bigl\{x\in \bbZ ^d \,:\, x\in b \text{ for some } b \in
 E (\o ) \bigr\}.
\end{align*}
Assuming the probability  $\bbQ( \o (b)>0)$ to be supercritical,
 the translation invariant Borel subset $\O_0\subset \O$ given by
the configurations $\o$ for which the graph $ G (\o)$ has a unique
infinite connected component (cluster) $\cC (\o) \subset  V (\o )$
has $\bbQ$--probability $1$ \cite{G}.
 Below, we denote by $\cE (\o)$ the bonds in $E(\o)$
connecting points of $\cC(\o)$ and we will often understand the fact
that $\o \in \O_0$.
%  We assume  that the
% network of bonds $b$ with positive conductance  has   a unique infinite
% cluster $\bbQ$ a.s.
% $\cC$, and call $\cE$ the family of associated  bonds.

%\begin{equation}
% \bbQ (\o (b) >0 ) >p_c \label{paranza1}\,.
%\end{equation}
% The random variable $\o(b)$
%is the conductance of  the bond $b$.

For later use, given $c>0$  we define the random field $\hat
\o_c=\bigl(\hat \o_c(b)\,:\, b \in \bbE_d\bigr)$  as
 \begin{equation}\label{paranza3}
\hat \o_c (b) = \begin{cases} 1 & \text{ if } \o(b)>c\,, \\
0 & \text{ otherwise}\,.
\end{cases}
\end{equation}
For $c=0$ we simply set $\hat \o:=\hat \o _0$.

\subsection{The zero range process on the infinite
cluster $\cC(\o)$}\label{lavinia_giorgia}  We fix a  nondecreasing
function $g:\bbN \rightarrow [0,\infty)$ such that $g(0)=0$,
$g(k)>0$ for all $k >0$ and
\begin{equation}
g^*:=\sup _{k \in \bbN} |g(k+1)-g(k)|<\infty\,.
\end{equation}

 Given a realization $\o$ of the environment, we
consider the  zero range process  $\eta _t $ on the graph
 $\cG(\o)= \bigl( \cC (\o), \cE (\o)\bigr) $ where a particle jumps from $x$ to $y$ with rate $g(\eta(x)) \o(x,y)$.
  This is the
 Markov process with paths $\eta(t)$ in the Skohorod space
$D\bigl( [0,\infty), \bbN  ^{\cC (\o) } \bigr)$  whose Markov
generator $\cL $ acts on local functions as
\begin{equation}\label{generatore}
\cL  f (\eta) = \sum _{e\in \cB} \sum _{\substack{x\in
\cC(\o)\,:\,\\
x+e \in \cC (\o) }  }g(\eta(x))  \o(x,x+e)
 \left( f(\eta ^{x,x+e})- f(\eta)\right)\,,
\end{equation}
where $\cB=\{ \pm e_1, \pm e_2, \dots, \pm e_d\}$, $e_1, \dots, e_d$
being the  canonical basis of $\bbZ^d$, and where in general
$$ \eta^{x,y} (z)  =
\begin{cases}
\eta(x)-1,&\text{ if }  z=x\,,\\
\eta(y)+1, &\text{ if }z=y\,,\\
\eta(z), &\text{ if } z \neq x,y\,.\\
\end{cases}
$$
We recall that a function $f$ is called local if $f(\eta)$  depends
only on $\eta(x)$  for a finite number of sites $x$.  Since
$\cC(\o)$ is infinite, the above process is well defined only for
suitable initial distribution. As discussed in \cite{A},
%Let us recall some results of
%\cite{A}.
% where an important role is plaid by the attractiveness of
%the process implied by the fact that $g$ is non decreasing.
% Let us recall some results of \cite{A},  where it is assumed only that $g(0)=0$ and $$g^*:= \sup_{k\geq 1}
%|g(k+1)-g(k) |<\infty\,.$$ There the author proves that
the process is well defined when the initial distribution has
support on configurations $\eta$ such that $\|\eta\|:=\sum _{x\in
\cC(\o) } \eta (x) a(x) <\infty$, $a(\cdot)$ being a strictly
positive real valued function on $ \cC (\o)$ such that
$$\sum _{x\in \cC(\o)} \sum_{e \in \cB} \o(x,x+e) a(x+e)\leq M a(x)$$
 for some positive constant $M$.

Given $\varphi \geq 0$, set  $Z(\varphi):= \sum_{k\geq 0} \varphi^k
/ g(k)!$ where $g(0)!=1$, $g(k)!=g(1) g(2) \cdots g(k)$ for $k\geq
1$. Since $Z(\varphi)$ is an increasing function  and $g(k)!\geq
g(1)^k$, there exists a critical value $\varphi_c\in (0, \infty]$
such that $Z(\varphi)<\infty$ if $\varphi < \varphi_c$ and
$Z(\varphi)=\infty$ if $\varphi>\varphi_c$. Then, for $0\leq
\varphi< \varphi_c$ we define
 $ \bar  \nu _\varphi$  as the product probability measure on $\bbN ^{\cC (\o)
}$ such that
$$ \bar  \nu_\varphi (\eta(x)=k) =\frac{1}{Z(\varphi) } \frac{\varphi^k}{
g(k)!}\,, \qquad k \in \bbN\,, \qquad x \in \cC (\o)\,.
$$
Taking for example $\a(x)= e^{-|x|}$ in the definition of $\|\eta\|$
one obtains that $\bar \nu_\varphi (\|\eta\|)<\infty$, thus implying
that  the zero range process is well defined whenever the initial
distribution is given by $\bar \nu _\varphi$ or by a probability
measure $\mu$ stochastically dominated by $\bar \nu _\varphi$. In
this last case, as proven in \cite{A}, by the monotonicity of $g$
one obtains that the zero range process $\eta_t$ starting from $\mu$
is stochastically dominated by the zero range process $\zeta_t$
starting from $\bar \nu_\varphi$, i.e. one can construct on an
enlarged probability space both processes $\eta _t$ and $\zeta _t$
s.t. $\eta_t(x) \leq \zeta_t (x)$ almost surely. Finally, we recall
that all measures $\bar \nu_\varphi$ are reversible  for the zero
range process and that $\bar \nu _\varphi (e^{\theta \eta (x)
})<\infty$ for some $\theta >0$, thus implying that  $\bar \nu
_\varphi (\eta(x)^k)<\infty$ for all $k \geq 0$ (cf. Section 2.3 of
\cite{KL}).

\medskip

As usually done, we assume that $\lim _{\varphi \uparrow \varphi_c}
Z(\varphi)=\infty $. Then, cf. Section 2.3 in \cite{KL}, the
function $R(\varphi):= \bar \nu _\varphi (\eta(0))$ is   strictly
increasing and gives a bijection from $[0,\varphi_c)$ to $[0,
\infty)$. Given $\rho \in [0,\infty)$ we will write $\varphi(\rho) $
for the unique value such that $R(\varphi)=\rho$. Then we set
\begin{equation}\label{pasqua0}
 \nu_\rho:= \bar \nu _{\varphi(\rho)}\,, \qquad
\phi(\rho):= \nu_{\rho }  ( g(\eta_x) ) \,\;\qquad x \in \cC (\o)
\,.
\end{equation}
As proven in Section 2.3 of \cite{KL}, $\phi$ is Lipschitz with
constant $g^*$.

\subsection{The hydrodynamic limit} Given an integer $N\geq 1$ and
a probability measure $\mu^N$ on   $\bbN^{\cC (\o)}$,  we denote by
$\bbP _{\o, \mu^N}$ the law of the  zero range  process with
generator $N^2 \cL$ (see   \eqref{generatore}) and with initial
distribution $\mu^N$ (assuming this dynamics to be  admissible). We
denote by
 $\bbE_{\o,\mu^N}$ the associated expectation. In order to state the
hydrodynamic limit, we define $B (\O)$ as the family of bounded
Borel functions on $\O$ and let  $\cD$ be the $d\times d $ symmetric
matrix characterized by the variational formula
\begin{equation}\label{varcar}
(a, \cD a) =\frac{1}{m} \inf _{\psi \in B (\O) }\left\{
 \sum _{e\in \cB_*} \int _\O \o (0,e)  ( a_e+ \psi (\t_e\o)-\psi (\o) ) ^2 \bbI _{0,e\in \cC (\o) }
  \bbQ ( d\o)\right\} \;,\;\;\forall a\in \bbR^d \,,
  \end{equation}
where $\cB_*$ denotes the canonical basis of $\bbZ^d$,
\begin{equation}\label{defm} m:= \bbQ \left(0 \in \cC (\o)
\right)\,  \end{equation} and the translated environment  $\t_e \o $
is defined as $\t_e \o (x,y)=\o(x+e, y+e)$ for all bonds $\{x,y\}$
in $\bbE_d$. In general, $\bbI_A$ denotes the characteristic
function of $A$.

The above matrix $\cD$ is the diffusion matrix of the random walk
among random conductances on the supercritical percolation cluster
and it equals the identity matrix multiplied by a positive constant
(see the discussion in  \cite{F} and references therein).

\begin{Th}\label{annibale}
For $\bbQ$--almost all environments $\o$ the following holds. Let
$\rho _0: \bbR ^d \rightarrow [0,\infty)$ be a bounded Borel
function and let $\{\mu _N\} _{N\geq 1}$ be a sequence  of
probability measures on $\bbN  ^{\cC (\o)} $ such that
 for all $\d>0$ and all continuous functions $G$ on $ \bbR^d $
with compact support (shortly $G \in C_c (\bbR ^d)) $, it holds
\begin{equation}\label{pasqua1}
\lim _{N \uparrow \infty} \mu^N \Big( \Big| N^{-d}  \sum _{x\in \cC
(\o) } G (x/N )\,  \eta (x) - \int _{\bbR ^d} G (x) \rho_0 (x) dx
\Big|>\d \Big )=0\,.
\end{equation}
Moreover, suppose that  there exist $\rho_0 ,\rho_*,C_0>0$  such
that $\mu^N$ is stochastically dominated by $\nu_{\rho_0}$ and the
entropy $H(\mu^N|\nu_{\rho_*} )$ is bounded by $C_0 N^d$.

 Then, for all $t>0$, $G \in C_c (\bbR ^d)$ and $\d>0$, it
 holds
\begin{equation}\label{pasqua2}
\lim _{N\uparrow \infty  } \bbP _{\o, \mu^N}  \Big ( \Big| N^{-d}
\sum _{x\in \cC (\o) } G (x /N)\,  \eta  _{ t} (x) - \int _{\bbR ^d}
G (x) \rho (x,t ) dx  \Big|>\d \Big )=0\,,
\end{equation}
where  $\rho: \bbR^d \times [0,\infty) \rightarrow \bbR$ is assumed to be  the
 unique weak solution of the heat equation
\begin{equation}\label{pasqua3}
\partial  _t \rho = m \nabla \cdot ( \cD \nabla  \phi(\rho/m))
\end{equation}
with boundary condition $\rho_0$ at $t=0$.
%  and where the symmetric matrix $\cD$ is variationally characterized by  (\ref{varcar}).
\end{Th}
%\club We recall that existence and uniqueness of the  weak solution
%of \eqref{pasqua3} follow from the Lipschitz property of $\phi$ and
%the results in  \cite{BC}.

We define the empirical measure $\p^N (\eta)$ associated to the
particle configuration $\eta$ as
$$ \p ^N (\eta) :=\frac{1}{N^d} \sum _{x \in \cC(\o)} \eta (x) \d _{x/N} \in
\cM (\bbR^d )\,, $$
 where $\cM  (\bbR^d)$ denotes the Polish space of non--negative
 Radon measures on $\bbR^d$ endowed with the vague topology (namely, $\nu_n \rightarrow \nu$ in $\cM$ if and only if
 $\nu_n(f)\rightarrow \nu (f)$ for each $f \in C_c(\bbR^d))$. We refer to the Appendix
of \cite{S} for a detailed discussion about  the space $\cM$ endowed
of  the vague topology.
% When considering the stochastic evolution of an initial
We write
 $\p^N_t$ for the empirical measure $\p^N (\eta_t)$, $\eta_t$ being the zero range process with
 generator $N^2 \cL$. Then condition
 \eqref{pasqua1} simply means that under $\mu^N$ the random measure
 $\p^N $ converges in probability to $\rho _0(x)dx$, while under
$\bbP_{\o, \mu^N}$ the random measure $\p^N_t$ converges in
probability to $\rho(x,t)dx$, for each fixed $t\geq 0$. In order to
prove the conclusion of Theorem \ref{annibale}   one only needs to
show that the law  of the random path $\p^N _\cdot \in D([0,T], \cM
)$ weakly converges to the delta distribution concentrated on the
path $[0,T]\ni t \rightarrow \rho(x,t) dx \in \cM$ (see
\cite{KL}[Chapter 5]). It is this stronger result that we prove
here.

\smallskip

Let us give some comments on our assumptions.  We have
restricted to increasing functions $g$ in order to assure
attractiveness and therefore that the dynamics is well defined
whenever the initial distributions are stochastically  dominated by
some invariant measure $\nu_{\rho_0}$. This simplifies also some
technical estimates. One could remove the monotone assumption on $g$
and choose other conditions assuring a well defined dynamics and
some basic technical estimates involved in the proof, which would be
similar to the ones appearing in \cite{KL}[Chapter 5].

\smallskip
The entropy bound $H(\mu^N|\nu_{\rho_*})\leq C_0 N^d$ is rather
restrictive. Indeed, given a locally Riemann integrable bounded
profile $\rho_0:\bbR^d\rightarrow [0,\infty)$, let $\mu^N$ be the
product measure on $\bbN^{\cC (\o)}$  with slowly varying parameter
associated to the profile $\rho_0/m$ at scale $N$. Namely, $\mu^N$
is the product measure on $\bbN^{\cC(\o)}$ such that
$$ \mu^N( \eta(x)=k) = \nu_{\rho_0 (x/N)/m } ( \eta(x)=k)\,.$$
Due to the ergodicity of $\bbQ$ condition \eqref{pasqua1} is
fulfilled and, setting $\rho' := \sup_x \rho_0 (x)$, $\mu^N$ is
stochastically dominated by $\nu_{\rho'/m}$. On the other hand, the
entropy $H(\mu^N|\nu_{\rho_*} )$ is given by
\begin{multline}
\frac{1}{N^d} H(\mu^N |\nu _{\rho_*} ) = \frac{1}{N^d} \sum _{x \in
\cC(\o) } \frac{\rho_0 (x/N)}{m}\Big[ \log \varphi\bigl(
\frac{\rho_0 (x/N)}{m}\bigr) -\log \varphi \bigl(\rho_*\bigr) \Big ]
\Big\}+
\\\frac{1}{N^d}\sum _{x \in \cC (\o) } \Big\{ \log Z\bigl(\varphi\bigl(\rho_*\bigr)) -\log Z\bigl(
\varphi \bigl(\frac{\rho_0(x/N)}{m} \bigr) \bigr)
\end{multline}
Hence, $H(\mu^N |\mu _{\rho_*} )\leq C_0 N^d$ only if $\rho_0$
approaches sufficiently fast the constant $m \rho_*$ at infinity.
All these technical problems are due to the infinite space. In order
to weaken the entropic assumption one should proceed as in
\cite{LM}. Since here we want to concentrate on the Moving Particle
Lemma, which is the real new problem, we keep our assumptions.

 % One can remove this assumption if one knows that
 %\begin{equation}\label{saratoga} \sup _{N\geq 1} \bbE_{\o, \mu^N} \Big( \sup_{0\leq t\leq T} N^{-d}
 %\sum _{x \in \cC(\o)} \eta _t (x) \Big)<\infty\,.
 %\end{equation}
%Note that if the above property is verified, then all $\mu^N$ must
%be concentrated on configurations with a finite number of particles.
%Hence, in this case the dynamics would be always well defined,
%without any restriction on $g$.

Finally, we have assumed uniqueness of the solution of  differential equation
 \eqref{pasqua3} with initial condition $\rho_0$. Results on uniqueness can be found in \cite{BC}, \cite{KL}[Chapter 5] and
 \cite{Va}.  Proceeding as in \cite{KL}[Section 5.7] and using the ideas developed below, one can prove
 that the limit points of the sequence $\{ \p^N_t \}_{t \in [0,T]}$ are concentrated on
 paths $t\rightarrow \rho(x,t)dx$ satisfying an energy estimate.

\section{Tightness of $\{ \p^N _t \}_{t \in [0,T]}$}
 As already mentioned, in order to reduce the proof of Theorem
\ref{annibale} to the Replacement Lemma one has to adapt the method
of the corrected empirical measure developed in \cite{GJ1} and after
that invoke some homogenization properties proved in \cite{F}. The
discussion in \cite{GJ1} refers to the zero range process on a
finite toroidal grid  and has to be modified in order to solve
technical problems due to the presence of infinite particles.

\smallskip

Given $N\in \bbN_+$, we define $\bbL_N$ as  the generator of the
  random walk among random conductances $\o$ on
the supercritical percolation cluster, after diffusive rescaling.
 More
precisely, we define $\cC_N(\o) =\{ x/N\,:\, x \in \cC(\o)\}$ and
set
\begin{equation}\label{scrivo}
 \bbL_N  f(x/N)=N^2 \sum _{\substack{e \in \cB:\\
x+e \in \cC(\o)} } \o(x,x+e)\Big\{f ((x+e)/N)-f(x/N) \Big\}
\end{equation}
for all $ x \in \cC(\o)$ and $ f : \cC_N(\o)\rightarrow \bbR$. We
denote by $\nu_\o^N$ the uniform  measure on   $\cC_N(\o)$ given by
$\nu_\o^N= N^{-d} \sum _{x \in \cC_N(\o)} \d_x$. Below  we will
think of the operator $\bbL_N$ as acting on $L^2 ( \nu^N_\o)$. We
write $(\cdot,\cdot)_{\nu^N_\o}$ and $\|\cdot \|_{L^2(\nu^N_\o)}$
for the scalar product and the norm  in $L^2(\nu^N_\o)$,
respectively. Note that $\bbL_N$   is a symmetric operator, such
that $(f, -\bbL_N f)_{\nu^N_\o}>0$ for each nonzero  function $f \in
L^2 ( \nu_\o^N) $. In particular, $\l \bbI - \bbL_N$ is invertible
for each $\l>0$. Moreover, it holds
\begin{equation*}
(f,-\bbL_N g )_{\nu^N_\o} =\frac{1}{2} \sum _{x \in \cC (\o)} \sum
_{\substack{e \in \cB:\\ x+e \in \cC(\o) }} \o (x,x+e) \bigl[
f(x+e)-f(x) \bigr]\cdot \bigl[ g (x+e)-g(x) \bigr]
\end{equation*}
for all functions $f,g \in L^2 (\nu^N_\o)$.
%Recall the definition of the matrix  $\cD$ given in \eqref{varcar}.
%Then the following homogenization result holds:
%\begin{Le} \cite{F} bla
%\end{Le}
 Given
$\l>0$, $G \in C^\infty_c (\bbR^d)$ and $N \in \bbN_+$ we define
$G^\l_N$ as the unique element of $L^2 ( \nu^N_\o)$ such that
\begin{equation}\label{gigia}
\l G^\l_N -\bbL_N G^\l_N = G^{\l}, \end{equation} where $G^\l$ is
defined as the restriction to $\cC_N(\o)$ of the function $\l G
-\nabla \cdot \cD \nabla G\in C^\infty_c (\bbR^d)$.

Let us collect some useful facts on the function $G^\l_N$:
\begin{Le}\label{meteo}
Fix $\l>0$. Then, for each $G \in C^\infty_c (\bbR^d)$ it holds
\begin{align}
&  (G^\l_N, -\bbL_N G^\l_N )_{\nu^N_\o} \leq
c(G,\l),\\
& \|G^\l_N\|_{L^1 (\nu^N_\o)}, \,\|G^\l_N\|_{L^2(\nu^N_\o)}\leq c(\l,G)\,,\\
& \|\bbL_N G^\l_N\|_{L^1 (\nu^N_\o)}, \,\|\bbL_N
G^\l_N\|_{L^2(\nu^N_\o)} \leq c(\l,G)\,,
\end{align}
for a suitable positive constant $c(\l, G)$ depending on $\l$ and
$G$, but not from $N$. Moreover, for $\bbQ$--a.s. conductance fields
$\o$ it holds
\begin{align}
\lim _{N\uparrow \infty} \| G^\l_N -G\|_{L^2(\nu^N_\o)} =0\,,
\qquad \forall G\in C^\infty_c (\bbR^d)\,, \label{sereno1}\\
\lim _{N\uparrow \infty} \| G^\l_N -G\|_{L^1(\nu^N_\o)} =0\,, \qquad
\forall G\in C^\infty_c (\bbR^d)\,. \label{sereno2}
\end{align}
\end{Le}
\begin{proof}
By taking the scalar product with $G^\l_N$  in
  \eqref{gigia} one obtains that
$$ \l \|G^\l_N\|^2_{L^2(\nu^N_\o)} +( G^\l_N,-\bbL_N  G^\l_N)_{\nu^N_\o}=
 ( G^\l_N, G^\l)_{\nu^N_\o} \leq \| G^\l_N\|_{L^2 (\nu^N_\o)}\| G^\l\|_{L^2
 (\nu^N_\o)}\,.
$$
Using that $\sup_N \| G^\l\|_{L^2
 (\nu^N_\o)}<\infty$, from the above expression one easily obtains the
 uniform upper  bounds on $(G^\l_N, -\bbL_N G^\l_N )_{\nu^N_\o}$ and $\|G^\l_N\|_{L^2(\nu^N_\o)}
 $. Since $\sup_{N\geq 1} \|G^\l \|_{ L^2(\nu^N_\o)}<\infty$, by
 difference one obtains the uniform upper  bound on
  $\|\bbL_N G^\l_N\|_{L^2(\nu^N_\o)}
 $.

  In order to estimate $\|G^\l_N\|_{L^1(\nu^N_\o)}
 $ let us write $p^N_t (x,y)$ for the probability  that the random walk
 on $\cC_N(\o)$ with generator $\bbL_N$ and starting point $x$ is
 at site $y$ at time $t$. Then, since the jump rates depend on the
 unoriented bonds, $p^N_t(x,y)=p^N_t (y,x)$. Since
 \begin{equation}\label{viaggio}
 G^\l _N(x)= \sum_{y \in \cC_N(\o)} \int_0
 ^\infty e^{-\l t} p^N_t (x,y) G^\l(y)\,,
 \end{equation}
  for all $x \in \cC_N(\o)$,
the above  symmetry allows  to
 conclude that
 \begin{multline}\label{cortona}
 \| G^\l_N\|_{L^1 (\nu^N_\o)}\leq  \frac{1}{N^d} \sum _{x ,y\in \cC_N(\o)} \int_0
 ^\infty e^{-\l t} p^N_t (x,y)| G^\l(x)|=\\
 \frac{1}{\l N^d} \sum _{x \in
 \cC_N(\o)}
 |G^\l (x)|\rightarrow \frac{1}{\l}\int _{\bbR^d} |G^\l (u)| du< \infty\,.
\end{multline}
Again, since   $\sup_{N\geq 1} \|G^\l \|_{ L^1(\nu^N_\o)}<\infty$,
by
 difference one obtains the uniform upper  bound on
  $\|\bbL_N G^\l_N\|_{L^1(\nu^N_\o)}$.

 The homogenization result \eqref{sereno1} follows from Theorem 2.4 (iii)
in \cite{F}.
%Strictly speaking, this last  would require $G^\l $
%having compact support. This assumption is assumed only to derive
%point (iii) of Theorem 2.4 from the previous points (i) and (ii)
%(see the last lines of Section 6). Technically the assumption
%assures that the function $u^0$ at the end of Section 6 is fast
%decaying at infinity. In our case, $u^0=G$ is even of compact
%support.
Finally, let us consider  \eqref{sereno2}.
Given $\ell >0$, using
Schwarz inequality, one can bound
\begin{multline*}
\|G^\l_N -G\|_{L^1(\nu^N_\o)} \leq\\ \| G^\l_N (u) \bbI( |u|> \ell)
\|_{L^1 (\nu^N_\o)}+\| G(u) \bbI( |u|> \ell) \|_{L^1 (\nu^N_\o)}+c
\ell^{d/2 }  \|G^\l_N-G\|_{L^2(\nu^N_\o)}^{1/2}\,.
\end{multline*}
 Since $G \in C^\infty_c (\bbR^d)$ the second term in
the r.h.s. is zero for $\ell$ large enough.  The last term in the
r.h.s. goes to zero due to \eqref{sereno1}. In order to conclude we
need to show that
\begin{equation}\label{sam}
\varlimsup_{\ell\uparrow \infty} \varlimsup_{N\uparrow \infty } \|
G^\l_N (u) \bbI( |u|> \ell) \|_{L^1 (\nu^N_\o)}=0\,.
\end{equation}
 Since $G^\l\in C^\infty_c
(\bbR^d)$ we can fix nonnegative  functions $F, f\in C^\infty_c
(\bbR^d)$ such that $- f\leq G^\l \leq F$. We call $F^\l_N$,
$f^\l_N$ the solutions in $L^2 (\nu^N_\o)$ of the equations
\begin{align*}
& \l F^\l_N-\bbL_N F^\l_N= F\,, \\
& \l f^\l_N-\bbL_N f^\l_N= f\,,
\end{align*}
respectively.
 From the integral representation \eqref{viaggio}  we derive that $F^\l_N$,
 $f^\l_N$ are nonnegative, and  that
$ -f^\l_N \leq G^\l_N \leq F^\l_N$ on $\cC_N(\o)$. In particular, in
order to prove \eqref{sam} it is enough to prove the same claim with
$F^\l_N$, $f^\l_N$ instead of $G^\l_N$. We give the proof for
$F^\l_N$, the other case is completely similar. Let us define $H \in
C^\infty(\bbR^d)$ as the unique solution in $L^2(dx)$ of the
equation
\begin{equation}\label{parigini}
\l H- \nabla \cdot \cD \nabla H= F\,.
\end{equation}
Again, by a suitable integral representation, we get that $H$ is
nonnegative.
 Applying Schwarz inequality, we can estimate
\begin{multline}\label{anton}
 \| F^\l_N (u) \bbI( |u|> \ell) \|_{L^1 (\nu^N_\o)}=\| F^\l_N
\|_{L^1 (\nu^N_\o)}-\| F^\l_N (u) \bbI( |u|\leq  \ell) \|_{L^1
(\nu^N_\o)}\leq \\
\| F^\l_N \|_{L^1 (\nu^N_\o)}-\| H (u) \bbI( |u|\leq  \ell) \|_{L^1
(\nu^N_\o)}+\| (H (u) - F^\l_N)\bbI( |u|\leq  \ell) \|_{L^1
(\nu^N_\o)}\leq\\ \| F^\l_N \|_{L^1 (\nu^N_\o)}-\| H (u) \bbI(
|u|\leq \ell) \|_{L^1 (\nu^N_\o)}
 +c \ell^{d/2}
\|F^\l_N-H\|_{L^2(\nu^N_\o)}^{1/2}\,.
\end{multline}
Since $F^\l_N$ and $F$ are nonnegative functions, when repeating the
steps in   \eqref{cortona} with $F^\l_N$, $F$  instead of $G^\l_N$,
$G^\l$ respectively, we get the the inequality is an equality and
therefore
$$\| F^\l_N \|_{L^1 (\nu^N_\o)}\rightarrow \l^{-1} \| F
\|_{L^1 (dx)}=\| H \|_{L^1 (dx)}\,.
$$
This observation, the above bound \eqref{anton} and Theorem 2.4
(iii) in \cite{F}  imply that
$$ \varlimsup_{N\uparrow \infty} \| F^\l_N (u) \bbI( |u|> \ell) \|_{L^1
(\nu^N_\o)}\leq \| H (u) \bbI( |u|> \ell) \|_{L^1 (\nu^N_\o)}
$$
At this point it is trivial to derive \eqref{sam} for $F^\l_N$.
\end{proof}
In the rest of this section, we will assume  that $\o$ is a good
conductance field, i.e. the infinite cluster  $\cC(\o)$ is
well--defined and $\o$ satisfies Lemma \ref{meteo}. We recall that
these properties hold $\bbQ$--a.s.

\medskip

 The first step in proving
the hydrodynamic limit consists in showing that the sequence of
processes $ \{ \p^N_t \}_{t \in [0,T]}$ is tight in the Skohorod
space $D([0,T], \cM)$. By adapting the proof of Proposition IV.1.7
in \cite{KL} to the vague convergence, one obtains that it is enough
to show that the sequence of processes $ \{ \p^N_t[G] \}_{t \in
[0,T]}$ is tight in the Skohorod space $D([0,T], \bbR)$ for all $G
\in C_c ^\infty (\bbR)$. A key relation between the zero range
process and the  random walk among random conductances is given by
\begin{equation}\label{terramatta}
N^2 \cL\left( \p^N (\eta)[G]\right)= \frac{1}{N^d} \sum_{x \in
\cC(\o)}  g(\eta (x) ) \bigl( \bbL_N G \bigr)(x/N)\,.
\end{equation}
%valid  for all functions $G : \cC_N(\o) \rightarrow \bbR$ such that
%the above sums are well--defined (e.g $G$ having finite support).
The check of \eqref{terramatta} is trivial and based on integration
by parts. At this point, due to the disorder given by the
conductance field $\o$, a second integration by parts as usually
done for gradient systems (cf. \cite{KL}[Chapter 5]) would be
useless since the resulting object would remain wild. A way to
overcome this technical problem is given by the method of the
corrected empirical measure: as explained below,  the sequence of
processes $ \{ \p^N_t[G] \}_{t \in [0,T]}$ behaves asymptotically as
$ \{ \p^N_t[G^\l_N] \}_{t \in [0,T]}$, thus the tightness of the
former follows from the tightness of the latter. We need some care
since the total number of particles can be infinite, hence it is not
trivial that the process $ \{ \p^N_t[G^\l_N] \}_{t \in [0,T]}$ is
well defined.

We start with a technical lemma which will be frequently used:
\begin{Le}\label{lucus}
Let   $H$ be a nonnegative function on $\cC_N(\o)$ belonging to $L^1
(\nu^N_\o)\cap L^2 (\nu^N_\o)$ and let $k\geq 0$. Then
 \begin{equation}\label{aprire}
  \bbP_{\o, \mu^N}\Big(
  \sup_{0 \leq t \leq T} \frac{1}{N^d} \sum _{x \in \cC(\o)}
\eta_t(x)^k H(x/N)>A \Big)\leq A^{-1}  c (k,\rho_0) \sqrt{
\|H\|^2_{L^1(\nu^N_\o)}+\|H\|^2_{L^2(\nu^N_\o)}}
\end{equation}
for all $A>0$.
\end{Le}
\begin{proof}
We use a maximal inequality for reversible Markov processes due to
Kipnis and Varadhan \cite{KV} (cf. Theorem 11.1 in Appendix 1 of
\cite{KL}).
 Let us set
 \begin{equation}
 \label{miele_mele}
 F(\eta)=\frac{1}{N^d} \sum _{x \in \cC(\o)}
\eta(x)^k H(x/N)\,,
\end{equation} supposing first that $H$ has bounded support.
Note that $F(\eta)\leq F(\eta')$ if $\eta(x)\leq \eta' (x)$ for all
$x \in \cC(\o)$. Hence by the stochastic domination assumption, it
is enough to prove \eqref{aprire} with $\bbP_{\o, \nu_{\rho_0}}$
(always referred to the diffusively accelerated process) instead of
$\bbP_{\o,\mu^N}$. We recall that $\nu_{\rho_0}$ is reversible
w.r.t. the the zero range process.  Moreover
\begin{equation}\label{cricri}
 \nu_{\rho_0} ( F^2) =\frac{1}{N^{2 d}} \sum
_{x,y \in \cC(\o)}  H(x/N) H(y/N) \nu_{\rho_0} ( \eta(x)^k
\eta(y)^k)\leq c(k,\rho_0) \| H\|^2_{L^1 (\nu^N_\o)}\,,
\end{equation}
while \begin{multline*} \nu_{\rho_0} (F, - N^2\cL F) =
\frac{N^2}{N^{2d} } \sum_{x,y\in \cC(\o)} H(x/N) H(y/N) \nu_{\rho_0}
( \eta (x)^k, -\cL \eta (y)^k)\leq \\c(k,\rho_0) \frac{N^2}{N^{2d} }
 \sum_{x\in \cC(\o)} \sum_{
\substack{y \in \cC(\o):\\ |x-y|=1} } H(x/N) H(y/N)\,.
\end{multline*}
Using the bound $H(x/N) H(y/N)\leq H(x/N)^2+ H(y/N)^2$ and the fact
that $d\geq 2$, we conclude that \begin{equation}\label{cracra}
\nu_{\rho_0} (F, - N^2\cL F)\leq c(k, \rho_0) \| H\|^2_{L^2
(\nu^N_\o)}\,.
\end{equation}
By the result of Kipnis and Varadhan it holds
\begin{equation}
  \bbP_{\o, \nu_{\rho_0}}\Big(
  \sup_{0 \leq t \leq T} \frac{1}{N^d} \sum _{x \in \cC(\o) }
\eta_t(x)^k H(x/N)>A \Big)\leq\frac{e}{A}
 \sqrt{ \nu_{\rho_0} ( F^2)+ T
 \nu_{\rho_0} (F, - N^2\cL F)}
 \,.
\end{equation}
 At this point the thesis follows from the above bounds
 \eqref{cricri} and \eqref{cracra}. In order to remove the
 assumption that $H$ is local, it is enough to apply the result to
 the sequence of functions $H_{n}(x):=H(x) \chi ( |x|\leq n)$ and
 then apply the Monotone Convergence Theorem as $n\uparrow \infty$.
 \end{proof}

 \begin{Rem} We observe that the arguments used in the proof of Lemma 4.3 in \cite{CLO} imply that, given a function
 $H$ of bounded support and defining $F$ as in \eqref{miele_mele}, it holds
 $$
 \bbE_{\o, \nu_{\rho_0} } \Big( \sup_{0\leq t\leq T} \bigl(F(\eta_t)-F(\eta_0) \bigr) ^2\Big)\leq c
 T  \nu_{\rho_0} (F, -N^2 \cL F)\,.
 $$
In particular, it holds
$$
\bbE_{\o, \nu_{\rho_0} } \Big( \sup_{0\leq t\leq T} F(\eta_t)  ^2\Big)\leq c \nu_{\rho_0} (F^2) +c T
\nu_{\rho_0} (F, -N^2 \cL F)\,.
 $$
Using the bounds \eqref{cricri} and
 \eqref{cracra}, the domination assumption and the Monotone Convergence Theorem, under the same assumption of Lemma
 \ref{lucus}  one obtains
 \begin{equation}\label{aprire_aia}
  \bbE_{\o, \mu^N}\Big(
  \sup_{0 \leq t \leq T}\Big[ \frac{1}{N^d} \sum _{x \in \cC(\o)}
\eta_t(x)^k H(x/N) \Big]^2 \Big)\leq  c (k,\rho_0)\Big[
\|H\|^2_{L^1(\nu^N_\o)}+\|H\|^2_{L^2(\nu^N_\o)} \Big]\,.
\end{equation}
Using afterwards the  Markov inequality, one concludes that
\begin{equation}\label{ghigginiaia}
  \bbP_{\o, \mu^N}\Big(
  \sup_{0 \leq t \leq T} \frac{1}{N^d} \sum _{x \in \cC(\o)}
\eta_t(x)^k H(x/N)>A \Big)\leq c (k,\rho_0)A^{-2}\Big[
\|H\|^2_{L^1(\nu^N_\o)}+\|H\|^2_{L^2(\nu^N_\o)}\Big]
\end{equation}
for all $A>0$. The use of \eqref{aprire} or \eqref{ghigginiaia} in the rest of the discussion is completely
equivalent.
  \end{Rem}

Due to Lemma \ref{lucus} and Lemma \ref{meteo} the process $ \{
\p^N_t[G^\l_N] \}_{t \in [0,T]}$ is well defined w.r.t $\bbE_{\o,
\mu^N}$. Let us explain why this process is a good approximation of
the process  $ \{ \p^N_t[G] \}_{t \in [0,T]}$:
\begin{Le}\label{bisio}
Let $G \in C^\infty_c (\bbR^d)$. Then, given $\d>0$, it holds
\begin{equation}
\lim_{N\uparrow \infty }\bbP_{\o, \mu^N}\Big(  \sup_{0\leq t \leq T}
\bigl| \p^N_t[G^\l_N]- \p^N_t[G]\bigr|>\d\Big)=0\,.
\end{equation}
\end{Le}
\begin{proof}
By Lemma \ref{lucus} we can bound the above probability by
$$ c(\rho_0) \d^{-1}\sqrt{  \| G^\l_N -G\|^2_{L^1(\nu^N_\o)}+ \| G^\l_N -G\|^2_{L^2(\nu^N_\o)}}\,.$$
At this point the thesis follows from Lemma \ref{meteo}.
\end{proof}
Due to the above Lemma, in order to prove the tightness of $ \{
\p^N_t[G] \}_{t \in [0,T]}$ it is enough to prove the tightness of $
\{ \p^N_t[G^\l_N] \}_{t \in [0,T]}$. Now we can go on with the
standard method based on martingales and Aldous criterion for
tightness (cf. \cite{KL}[Chapter 5]), but again we need to handle
with care our objects due to the risk of explosion. We fix a good
realization $\o$  of the conductance field.  Due to  Lemma
\ref{meteo}, Lemma  \ref{lucus} and the bound $g(k) \leq g^* k$, we
conclude that
  the process $\{M^{N}_t\}_{0\leq t \leq T}$ where
\begin{equation}\label{pranzetto}
M^{N}_t(G):= \p^N_t (G^\l_N)- \p^{N}_0 (G^\l_N ) -\int _0^t
\frac{1}{N^d} \sum_{x \in \cC (\o)} g( \eta _{s} (x) ) \bbL_N G^\l
_N (x/N) ds\,,
\end{equation}
is well defined  $\bbP_{\o, \mu^N}$--a.s.
% Let us prove the tightness
%of $\{M^N_t (G) \}_{t \in [0,T]}$ using Aldous criterium (cf.
%Proposition 1.2 and Proposition 1.6 in Section 4 of \cite{KL}):
\begin{Le}\label{aldo}
%The process $\{M^N_t (G) \}_{t \in [0,T]}$ is tight in
%$D([0,T],\bbR)$. Moreover,
Given $\d>0$,
\begin{equation}\label{volo} \lim_{N\uparrow \infty} \bbP_{\o,
\mu^N} \Big( \sup_{0 \leq t \leq T} |M^{N}_t (G) |\geq \d \Big)=0\,.
\end{equation}
\end{Le}
\begin{proof}
Given $n\geq 1$, we define the cut--off function $G^\l_{N,n} :
\cC_N(\o) \rightarrow \bbR$ as $ G^\l_{N,n} (x) =G^\l_N (x) \bbI
(|x|\leq n)$.
 Then $G^\l_{N,n}$
is a local function and by the results of \cite{A} (together with
the stochastic domination assumption) we know that
$$
M^{N,n}_t(G):= \p^N_t (G^\l_{N,n} )- \p^{N}_0 (G^\l_{N,n} ) -\int
_0^t \frac{1}{N^d} \sum_{x \in \cC (\o)} g( \eta _{s } (x) ) \bbL_N
G^\l _{N,n} (x/N) ds
$$
is an $L^2$--martingale of quadratic variation
$$ < M^{N,n} _t (G) >= \int _0 ^t \frac{N^2}{N^{2d}} \sum _{x \in \cC(\o)} \sum _{
\substack{ y \in \cC(\o): \\|x-y|=1}} g (\eta_{s } (x) )\o(x,y)
\bigl[ G^\l_{N,n}(y/N) -G^\l_{N,n} (x/N) \bigr]^2ds\,.
$$
Note that, by the stochastic domination assumption
 and the bound $g(k) \leq g^* k$,
\begin{multline}\label{tenda}
\bbE_{\o, \mu^N} \bigl( <M^{N,n}_t (G)
>\bigr)\leq \\g^* \int _0 ^t \frac{N^2}{N^{2d}} \sum _{x \in \cC(\o)} \sum _{
\substack{ y \in \cC(\o): \\|x-y|=1}}\bbE_{\o, \nu_{\rho_0} }
\bigl[\eta_s (x) \bigr]\o(x,y) \bigl[ G^\l_{N,n}(y/N) -G^\l_{N,n}
(x/N) \bigr]^2ds \\
=g^* \rho_0 t N^{-d}( G^\l_{N,n}, -\bbL _N G^\l _{N,n}
)_{\nu^N_\o}\,.
\end{multline}
By Doob's inequality and \eqref{tenda}, we conclude that
\begin{multline}\label{dandrea}
 \bbP_{\o, \mu^N} \Big( \sup_{0 \leq t \leq T} |M^{N,n}_t (G) |\geq
\d \Big)\leq \frac{c}{ \d^2} \bbE_{\o, \mu^N} \bigl( |M^{N,n}_T (G)
|^2 \bigr) \leq\\
  \frac{ c g^* T \rho_0}{
\d^2 N^d }( G^\l_{N,n}, -\bbL _N G^\l _{N,n} )_{\nu^N_\o}\leq
\frac{c' g^* T \rho_0}{ \d^2 N^d }\,.
\end{multline}
Above  we have used that    $\lim_{n\uparrow \infty }( G^\l_{N,n},
-\bbL _N G^\l _{N,n} )_{\nu^N_\o}=( G^\l_N, -\bbL _N G^\l _N
)_{\nu^N_\o} \leq c(\l)$ (see Lemma \ref{meteo}).

 The above process $\{M^{N,n}_t(G)\}_{t\in [0,T]}$ is a
good approximation of $\{M^{N}_t(G)\}_{t\in [0,T]}$ as $n\uparrow
\infty$. Indeed, it holds
\begin{equation}\label{ale}
\lim _{n\uparrow \infty} \bbP_{\o,\mu^N} \Big( \sup _{0\leq t \leq
T}|M^{N,n}_t(G)- M^N_t(G)|>\d\Big)=0\,, \qquad \d>0\,.
\end{equation}
Indeed, since $\| G^\l _{N,n} -G^\l_N\|_{ L^1(\nu^N_\o),
L^2(\nu^N_\o) } $ and $\| \bbL_N G^\l _{N,n} -\bbL_N G^\l_N\|_{
L^1(\nu^N_\o), L^2(\nu^N_\o) } $ converge to zero as $N\uparrow
\infty$ and since $g(k) \leq g^* k$, the above claim follows from
Lemma \ref{lucus}.

\medskip
 At this point,
\eqref{ale} and \eqref{dandrea} imply \eqref{volo}.
\end{proof}

Let us prove the tightness of $\{\p^N_t [G^\l_N ] \}_{t \in [0,T]}$
using Aldous criterion (cf. Proposition 1.2 and Proposition 1.6 in
Section 4 of \cite{KL}):
\begin{Le}\label{formentin}
Given $G \in C^\infty_c (\bbR^d)$, the sequence of processes
$\{\p^N_t [G^\l_N ] \}_{t \in [0,T]}$ is tight in $D([0,T], \bbR)$.
\end{Le}
\begin{proof} Fix $\theta >0$ and
uppose  that $\t$ is a stopping time w.r.t. the canonical filtration
bounded by $T$. With some abuse of notation we write $\t+\theta $
for the quantity $\min \{\t+\theta,T\}$. Then, given $\e>0$, by
Lemmata \ref{meteo} and \ref{lucus},
\begin{multline}\label{fresco} \bbP_{\o, \mu^N}\bigl( \Big |\int _\t^{\t+\theta}
\frac{1}{N^d} \sum_{x \in \cC
(\o)} g( \eta _{s} (x) ) \bbL_N G^\l _N (x/N) ds\Big|>\e\Big)\leq\\
\bbP_{\o, \mu^N}\bigl( \theta g^* \sup_{s \in [0,T] }  \frac{1}{N^d}
\sum_{x \in \cC (\o)} \eta _{s} (x)  |\bbL_N G^\l _N (x/N)| ds
>\e\Big)\leq C g^* \theta /\e\,.
\end{multline}
In particular,
\begin{equation}\label{ciclo1}
\lim_{\g \downarrow 0} \limsup_{N\uparrow \infty} \sup_{\t, \theta
\in [0, \g]} \bbP_{\o, \mu^N}\bigl( \Big |\int _\t^{\t+\theta}
\frac{1}{N^d} \sum_{x \in \cC (\o)} g( \eta _{s} (x) ) \bbL_N G^\l
_N (x/N) ds\Big|>\e\Big)=0 \,.
\end{equation}
An estimate similar to \eqref{fresco} implies that
\begin{equation}\label{ciclo2}
\lim _{\e \uparrow \infty }\sup_{N\geq 1}   \bbP_{\o, \mu^N}\bigl(
\Big |\int _0 ^t \frac{1}{N^d} \sum_{x \in \cC (\o)} g( \eta _{s}
(x) ) \bbL_N G^\l _N (x/N) ds\Big|>\e\Big)=0 \,.
\end{equation}
Let us now come back to Lemma \ref{aldo}. Let $\t, \theta$ as above.
 Then
 \begin{equation}\label{ciclo3} \limsup _{N\uparrow \infty } \bbP_{\o, \mu^N}\bigl(  |
M^N_{\t+\theta}-M^N_\t |>\e\Big)\leq  \limsup_{N\uparrow \infty}
\bbP_{\o, \mu^N}\bigl( \sup_{0\leq s \leq T} \ | M^N_s |>\e/2\Big)=0
\end{equation} Collecting \eqref{ciclo1}, \eqref{ciclo2} and
\eqref{ciclo3}, together with Lemma \ref{aldo}, we conclude that
\begin{align}
& \lim_{\g \downarrow 0} \limsup_{N\uparrow \infty} \sup_{\t, \theta
\in [0, \g]} \bbP_{\o, \mu^N}\bigl( \bigl | \p^N_{\t+\theta}
[G^\l_N]- \p^N_\t [G^\l_N] \bigr|>\e\Big)=0 \,, \label{uuu1}\\
& \lim _{\e \uparrow \infty }\sup_{N\geq 1}   \bbP_{\o, \mu^N}\bigl(
| \p^N_t [G^\l_N]>\e\Big)=0\label{uuu2}\,.
\end{align}
Aldous criterion for tightness allows to derive the thesis from
\eqref{uuu1} and \eqref{uuu2}.

\end{proof}

Let us come back to \eqref{pranzetto} and investigate the integral
term there. The following holds:
\begin{Le}\label{padova}
Let $ I(t):= \int_0 ^tN^{-d} \sum_{x \in \cC(\o)} g( \eta_s(x))
\bigl( G^\l_N(x/N) - G(x/N) \bigr)ds $. Then, for all $\d>0$,
\begin{equation}\label{fiorello}
\lim _{N\uparrow \infty} \bbP_{\o, \mu^N} \bigl( \sup_{0\leq t \leq
T} |I(t) |>\d\bigr)=0 \,.
\end{equation}
\end{Le}
\begin{proof}
Since $g(k)\leq g^* k$ and by Schwarz inequality we can bound
$$ I(t) \leq J:=T g^*\| G^\l_N - G\|_{L^1(\nu^N_\o) }^{1/2}   \sup_{0\leq s\leq T} \{N^{-d} \sum_{x \in \cC(\o)}
\eta_s(x)^2 \bigl| G^\l_N(x/N) - G(x/N) \bigr|\}^{1/2}\,.
$$
Using the stochastic domination assumption  and applying Lemma
\ref{lucus} we obtain
\begin{multline}
 \bbP_{\o, \mu^N} \bigl( \sup_{0\leq t \leq
T} |I(t) |>\d\bigr)\leq \bbP_{\o, \nu_{\rho_0} } (J>\d) \leq
\\(1/\d)^2  T^2 ( g^*)^2  \| G^\l_N - G\|_{L^1(\nu^N_\o) } \sqrt{
\| G^\l_N - G\|_{L^1(\nu^N_\o) }^2+\| G^\l_N - G\|^2_{L^2(\nu^N_\o)
}}\,.
\end{multline}
The thesis now follows by applying Lemma \ref{meteo}.
\end{proof}
We are finally arrived at the conclusion. Indeed, due to Lemma
\ref{bisio} and Lemma \ref{formentin} we know that the sequence of
processes $\{\p^N_t\}_{t \in [0,T]}$ is tight in the Skohorod space
$D([0,T], \cM)$. Moreover, starting from the identity
\eqref{pranzetto},  applying Lemma \ref{aldo}, using  the identity
\eqref{gigia} which  equivalent to $$\bbL_N G_N^\l=\l G^\l_N- G^\l=
\l (G^\l_N-G)+ \nabla \cdot \cD \nabla G\,,$$ and finally invoking
Lemma \ref{padova} we conclude that, fixed a good conductance field
$\o$, for any $G \in C^\infty_c(\bbR^d)$ and for any $\d>0$
\begin{equation}\label{superiore}
\lim_{N\uparrow \infty} \bbP_{\o, \mu^N} \Bigl ( \sup_{0\leq t \leq
T} \Big|\p ^N_t(G)-\p^N_0(G)-\int_0 ^t \frac{1}{N^d} \sum _{x \in
\cC (\o)} g ( \eta_s (x) ) \nabla \cdot \cD \nabla G (x/N)ds
\Big|>\d\Big)=0\,.
\end{equation}
Using the stochastic domination assumption it is trivial to prove
that any limit point of the sequence $\{\p^N_t\}_{t \in [0,T]}$ is
concentrate on trajectories $\{ \p_t \}_{t \in [0,T]}$ such that
$\p_t $ is absolutely continuous w.r.t. to the Lebesgue measure.
Moreover, in order to characterize the limit points as solution of
the differential equation \eqref{pasqua3} one would need non only
\eqref{superiore}. Indeed, it is necessary to prove that, given $\o$
good,  for each function $G \in C^\infty_c \bigl([0,T]\times
\bbR^d\bigr) $ it holds
\begin{multline}\label{superiorebis}
\lim_{N\uparrow \infty} \bbP_{\o, \mu^N} \Bigl ( \sup_{0\leq t \leq
T} \Big|\p ^N_t(G_t)-\p^N_0(G_0 )-\\\int_0 ^t \frac{1}{N^d} \sum _{x
\in \cC (\o)} g ( \eta_s (x) ) \nabla \cdot \cD \nabla G_s  (x/N)ds-
\int_0^t \p^N_s (\partial _s G_s )ds   \Big|>\d\Big)=0\,,
\end{multline}
where $G_s(x):= G(s,x)$.  One can easily recover
\eqref{superiorebis} from the same estimates used to get
\eqref{superiore} and suitable approximations of $G$ which are
piecewise linear in $t$ as in the final part of Section 3 in
\cite{GJ1}. In order to avoid heavy notation will continue the
investigation of \eqref{superiore} only.

\section{The Replacement Lemma}\label{rep_lemma}

 As consequence of the discussion in the previous section,
 in order to
prove the hydrodynamical limit stated in Theorem \ref{annibale} we
only need
 to control the term
\begin{equation}\label{sale} \int_0 ^t \frac{1}{N^d} \sum _{x \in \cC (\o)} g ( \eta_s (x)
) \nabla \cdot \cD \nabla G (x/N)ds \,.
\end{equation}
To this aim we  first introduce some  notation. Given a family of
parameters $\a_1, \a_2, \dots , \a_n$, we will write
$$
\limsup_{\a_1 \rightarrow a_1, \a_2 \rightarrow a_2, \dots, \a_n
\rightarrow a _n} $$ instead of
$$
\limsup_{\a_n \rightarrow a_n} \,\limsup_{\a_{n-1} \rightarrow
a_{n-1} } \cdots \, \limsup_{\a_1 \rightarrow a_1}\,.
$$

 Below, given $x \in
\bbZ^d$ and $k \in \bbN$, we write $\L_{x, k}$ for the box
$$ \L_{x,k}:= x+ [-k,k]^d \cap \bbZ^d \,,$$
and we write $\eta^k(x)$ for the density
$$ \eta^k(x) := \frac{1}{(2k+1)^d} \sum _{y\in  \L_{x,k}\cap \cC(\o)
} \eta (y) \,.
$$
If $x=0$ we simply write $\L_k$ instead of $\L_{0,k}$.

 Then,  we
claim that for $\bbQ$--a.a. $\o$, given $G \in C_c (\bbR^d) $,
$\d>0$ and a sequence $\mu^N$ of probability measures on $\bbN ^{\cC
(\o)} $  stochastically dominated by some $\nu_{\rho_0}$ and such
that $H(\mu^N|\nu_{\rho_*})\leq C_0 N^d$, it holds
\begin{multline}\label{codroipo}
\limsup_{N\uparrow \infty,\e\downarrow 0} P_{\o, \mu^N} \Big ( \Big
| \int_0^t  \frac{1}{N^d} \sum _{x \in \cC (\o) } g \bigl( \eta_s
(x) \bigr) G(x/N) ds - \\  \int_0 ^t \frac{m}{N^d} \sum_{x\in
\bbZ^d} \phi \bigl ( \eta _s ^{\e N}  (x) /m  \bigr)G(x /N) ds
\Big|>\d
  \Big)=0\,.
\end{multline}
Let us first assume the above claim and explain how to conclude,
supposing for simplicity of notation that $\e N \in \bbN$.

\smallskip

 Given $u \in \bbR^d$ and $\e>0$, define  $\iota_{u,\e}:=
(2\e)^{-d} \bbI \{ u\in \overline{\overline{}}[-\e,\e]^d\}$. Then
the integral $\p^N\bigl[\iota_{x/N,\e} \bigr]$, $ x \in \bbZ^d$, can
be written as \begin{equation}\label{bo}
 \p^N \bigl[ \iota _{x/N,
\e} \bigr]=\frac{(2\e N +1)^d }{(2\e N)^d}\eta ^{\e N}(x) \,.
\end{equation}
 Let us define
\begin{equation}\label{sale1}
 \int_0 ^t \frac{m}{N^d} \sum _{x \in \bbZ^d } \phi ( \p_s^N  [\iota _{x/N, \e}/m
 ]
) \nabla \cdot \cD \nabla G (x/N)ds \,.
\end{equation}
Then, due to \eqref{bo} and  since $\phi$ is Lipschitz with constant
$g^*$, we can estimate from above the difference between
\eqref{sale1} and the second integral term in \eqref{codroipo} with
$G$ substituted by $ \nabla \cdot \cD \nabla G$ as
\begin{equation}\label{cortocircuito}
\frac{(2\e N +1)^d - (2\e N)^d }{(2\e N)^d} \int_0^t \frac{1}{N^d}
\sum_x \eta^{\e N}_s (x) \bigl|\nabla \cdot \cD \nabla G
(x/N)\bigr|ds
\end{equation}
Since the integral term in \eqref{cortocircuito} has finite
expectation w.r.t $\bbP_{\o, \nu_{\rho_0}}$ and therefore also
w.r.t. $\bbP_{\o, \mu^N}$,  we conclude that the above difference
goes to zero in probability w.r.t. $\bbP_{\o, \mu^N}$. At this point
the conclusion of the proof of Theorem \ref{annibale} can be
obtained by the same arguments used in
 \cite{KL}[pages 78,79].

\medskip

Let us come back to our claim. Since $$\sum _{x \in \cC (\o)}
g(\eta_s (x) ) G(x/N)=\sum _{x \in \bbZ^d} g (\eta_s (x) ) \bbI (x
\in \cC(\o)) G(x/N) \,, $$  by a standard integration by parts
argument  and using that $G \in C_c ^\infty (\bbR^d)$ one can
replace the first integral in \eqref{codroipo} by
$$ \int _0^t\frac{1}{N^d} \sum_{x \in \bbZ^d}   \Big[\frac{1}{ (2\ell +1)^d} \sum _{y:y \in
\L_{x,\ell} \cap \cC(\o) } g\bigl( \eta_s(y) \bigr)\Big] G(x/N)\,.
$$
Then the claim \eqref{codroipo} follows from
\begin{Le} (\textrm{{\bf Replacement Lemma}})
For $\bbQ$--a.a. $\o$, given $\d>0$, $M\in \bbN $ and  given a
sequence of probability measures $\mu^N$ on  $\bbN ^{\cC (\o)} $
stochastically dominated by some $\nu_{\rho_0}$ and such that
$H(\mu^N|\nu_{\rho_*})\leq C_0 N^d$, it holds
\begin{equation}\label{russia}
\limsup_{N\uparrow \infty,\e \downarrow 0 }\bbP_{\o, \mu^N} \Big[
\int _0 ^t \frac{1}{N^d} \sum _{x \in \L_{MN} } V_{\e N} (\tau _x
\eta_s, \tau _x \o )ds>\d \Big]=0\,,
\end{equation}
where
$$ V_\ell ( \eta, \o) = \Big | \frac{1}{ (2\ell +1)^d} \sum _{y:y \in
\L_\ell \cap \cC(\o) } g\bigl( \eta(y) \bigr) -m  \phi \bigl( \eta
^\ell (0)/m \bigr) \Big|\,.
$$
\end{Le}

Let us define $\Upsilon_{C_0,N}$ as the set of measurable functions
$f: \bbN^{\cC (\o) } \rightarrow [0,\infty)$ such that i) $\nu
_{\rho_*} (f)=1$, (ii) $\cD(f):= \nu_{\rho_*} (\sqrt{f}, -\cL
\sqrt{f} )  \leq C_0 N^{d-2}$ and  (iii) $f d\nu_{\rho_*}$ is
stochastically dominated by $\nu_{\rho_0} $ (shortly, $f d
\nu_{\rho_*} \prec d \nu_{\rho_0}$).
 Using the assumption $H(\mu^N| \nu _{\rho_*})\leq C_0 N^d$ and entropy production arguments as in \cite{KL}[Chapter 5],
in order to prove the Replacement Lemma it is enough to show that
for $\bbQ$--a.a. $\o$, given  $\rho_0, \rho_*, C_0,M>0$, it holds
\begin{equation}\label{rap}
\limsup _{N\uparrow \infty, \e \downarrow 0} \sup_{f\in
\Upsilon_{C_0,N} } \int \frac{1}{N^d} \sum _{x \in \L_{MN} } V_{\e
N}  (\t_x \eta, \t_x \o ) f(\eta) \nu _{\rho_*} (d \eta ) = 0 \,.
\end{equation}
Trivially, since $\nu_{\rho_1}$ stochastically dominates $
\nu_{\rho_2}$ if $\rho_1 >\rho_2$, it is enough to prove that, given
$\rho_0,  \rho_*, C_0,M>0$, for $\bbQ$--a.a. $\o$ \eqref{rap} is
verified.
%Reasoning as in \cite{KL}[pages 83,84] where $g(\eta(z) )$ and $\bbT
%_N ^d $ have  to be replaced by $g(\eta(z) ) \chi (z \in \cC (\o) )$
%and $\cC (\o)$ respectively, it is simple to derive
We claim that the above result follows from the the One Block and
the Two Blocks estimates:
\begin{Le} (\textrm{{\bf One block estimate}})  Fix  $ \rho_0,\rho_*, C_0,M>0$. Then,
for $\bbQ$--a.a. $\o$ it holds
\begin{equation}\label{1block}
\limsup _{N\uparrow \infty, \ell \uparrow \infty}  \sup_{f\in
\Upsilon_{C_0,N} } \int \frac{1}{N^d} \sum _{x \in \L_{MN} } V_\ell
(\t_x \eta, \t_x \o ) f(\eta) \nu _{\rho_*} (d \eta ) = 0 \,.
\end{equation}
\end{Le}
\begin{Le}
 \textrm{{(\bf Two blocks estimate)}} Fix  $ \rho_0,\rho_*, C_0,M>0$. Then,
for $\bbQ$--a.a. $\o$ it holds
\begin{multline}\label{2blocks}
\limsup _{N\uparrow \infty,  \e \downarrow 0, \ell \uparrow \infty}
 \sup_{f\in
\Upsilon_{C_0,N} }\\ \int \frac{1}{N^d} \sum _{x \in \L_{MN} }
\Big[\frac{1}{(2\e N+1)^d } \sum _{ y \in \L_{x, \e N} } \bigl| \eta
^\ell (y) - \eta ^{\e N} (x) \bigr|\Big] f(\eta) \nu _{\rho_*} (d
\eta ) = 0 \,.
\end{multline}
\end{Le}
We point out that the form of the Two Blocks Estimate is slightly
weaker  from the one in \cite{KL}[Chapter 5], nevertheless it is
strong enough to imply, together with the One Block Estimate,
equation \eqref{rap}. Indeed, let us define $a(y):= \bbI (y \in \cC
(\o)) $ and
\begin{align}
& I_1 (\eta):=\frac{1}{N^d} \sum _{x \in \L_{MN} } \left| Av _{y \in
\L_{x, \e N} } g( \eta (y) ) a (y) - Av _{y \in \L_{x, \e N } }
Av_{z\in \L_{y, \ell } }  g( \eta (z) ) a (z)\right|\,,
\\
& I_2 (\eta):=  \frac{1}{N^d} \sum _{x \in \L_{M N} }\left| Av_{y
\in \L_{x, \e N } } \Bigl( Av _{z \in \L_{y,  \ell } } g( \eta (z) )
a (z)- m \phi (
\eta^\ell (y) /m ) \Bigr)\right|  \,,\\
& I_3 (\eta):=\frac{1}{N^d} \sum _{x \in \L_{M N} } \left| Av_{y \in
\L_{x, \e N } } m\phi ( \eta^\ell (y) /m) - m\phi ( \eta^{\e N }
(x)/m ) \right)\,,
% & I_4 (\eta):=\frac{m}{N^d} \Big| \sum _{x \in \L_{M N}
%}\Big[ Av_{y \in \L_{x, \e N } } \phi ( \eta^\ell (y) /m) - \phi (
%
%\eta^{\e N } (x)/m ) \Big] \Big|\,,
\end{align}
 where $Av$ denotes
the standard average. Then
\begin{multline}\label{acqua}\frac{1}{N^d} \sum_{x \in \L_{M N}}
V_{\e N} (\tau _x \eta, \tau _x \o )=\\ \frac{1}{N^d} \sum_{x \in
\L_{M N}}\Big| Av_{y \in \L_{x, \e N}} g (\eta(y) )a(y) - m \phi (
\eta^{\e N}(x)/m) \Big|\leq (I_1+I_2+I_3)(\eta)\,.
\end{multline}

Let us explain a simple bound that will be frequently used below,
often without any mention. Consider a family of numbers  $b(x)$, $x
\in \bbZ^d$. Then, taking $L, \ell >0$ we can write
\begin{multline*} Av _{x \in \L_L} b(x)- Av_{x\in \L_L} Av_{u \in
\L_{x,\ell} }b(u)
=\\
\frac{1}{|\L_L| } \sum _{x \in \L_{L+\ell}} b(x) \bigl( \bbI(x\in
\L_L) - \frac{1}{|\L_\ell|}\sharp\{u \in \L_L: |x-u|_\infty\leq
\ell\} \bigr).
\end{multline*}
In particular, it holds
\begin{equation}\label{semplice}
\bigl| Av _{x \in \L_L} b(x)- Av_{x\in \L_L} Av_{u \in
\L_{x,\ell}}b(u) \bigr|\leq  \frac{1}{|\L_L| } \sum _{x \in
\L_{L+\ell}\setminus \L_{L-\ell} } |b(x)| \,.
\end{equation}
Due to the above bound  we conclude that $$I_1 (\eta) \leq \frac{g^*
}{N^d (2\e N+1)^d}  \sum _{x \in \L_{MN} } \sum _{u \in \L_{x, \e
N+\ell}\setminus \L_{x, \e N - \ell} } \eta (u) a(u) \,.
$$
In particular, using that $f d\nu_{\rho_*} \prec d\nu_{\rho_0}$, we
conclude that $ \int I_1 (\eta) f(\eta) \nu_{\rho_*}(d \eta) \leq  c
\ell/(\e N)\,. $
 The second term $I_2(\eta)$ can be estimated for $\e\leq 1$ as
$$ I_2(\eta)\leq \frac{1}{N^d} \sum _{x \in \L_{M N} } Av_{y
\in \L_{x, \e N } }  V_{\ell} (\t_z \eta, \t_z \omega)\leq
\frac{1}{N^d} \sum _{x \in \L_{(M+1) N} } V_{\ell} (\t_x \eta, \t_x
\omega) \,.$$ Due to the One block estimate one gets that
$$\limsup _{N\uparrow \infty,  \e \downarrow 0, \ell \uparrow \infty}\sup _{f \in \Upsilon_{C_0,N} }\int
I_2 (\eta) f(\eta) \nu_{\rho_*} (d\eta)=0\,.
$$
The same  result holds also for $I_3(\eta) $ due to  the Lipschitz
property of $\phi$ and the Two Blocks estimate. The above
observations together with \eqref{acqua} imply \eqref{rap}.

%Note that for the
%above arguments the boundedness of $g$ was not essential (in more
%general cases one could use the stochastic domination assumptions in
%order to bound the expectations).

%\section{Percolation results}
%In this section we collect some results concerning bond percolation
%which will be

\section{Proof of the Two Blocks Estimate}

 For simplicity of notation we set $\ell_*=(2\ell+1)^d$ and we take $M=1$ (the general case can be treated
similarly). Moreover, given $\D\subset \bbZ^d$, we write $\cN(\D)$
for the number of particles in the region  $\D$, namely
$\cN(\D):=\sum _{x \in \D\cap \cC(\o)} \eta (x)$.

Let us set
\begin{align}
& A(\eta):= \frac{1}{N^d} \sum _{x \in \L_{N} } Av_{ y \in \L_{x, \e
N} } \bigl| \eta ^\ell (y) - \eta ^{\e N}
(x) \bigr|\,,\\
& B(\eta):= \frac{1}{N^d } \sum _{x \in \L_{N} }  Av_{ y \in \L_{x,
\e N} } Av_{z \in \L_{x, \e N }}\bigl| \eta ^\ell (y) - \eta ^\ell
(z)
\bigr|\,,\\
 & C(\eta):=\frac{1}{N^d |\L_{\e N}|} \sum _{x \in \L_{N} }  Av_{ y \in
\L_{x, \e N}
} \sum_{\substack{ z \in \L_{x, \e N }:\\
|z-y|_\infty \geq 2\ell } }\bigl| \eta ^\ell (y) - \eta ^\ell (z)
\bigr|\,.
\end{align}
 Since due to \eqref{semplice} $$\eta ^{\e N}
(x)= Av_{z \in \L_{x, \e N}} \eta ^\ell (z) + \cE\,, \qquad |\cE|
\leq c |\L_{\e N}|^{-1} \sum _{u \in \L_{x, \e N+\ell} \setminus
\L_{x, \e N -\ell} } \eta (u)\,,
$$
using that  $f d\nu_{\rho_*} \prec d\nu_{\rho_0}$, in order to prove
the Two Blocks estimate we only need to show that $$ \limsup
_{N\uparrow \infty,  \e \downarrow 0, \ell \uparrow \infty}\sup _{f
\in \Upsilon_{C_0,N} }\int B (\eta) f(\eta) \nu_{\rho_*}
(d\eta)=0\,.
$$
Since
$$ B(\eta) \leq C(\eta)+\frac{1}{N^d |\L_{\e N}|} \sum _{x \in \L_{N} }  Av_{ y \in
\L_{x, \e N}
} \sum_{\substack{ z \in \L_{x, \e N }:\\
|z-y|_\infty \leq  2\ell } }\bigl( \eta ^\ell (y) +\eta ^\ell (z)
\bigr)\,,
$$
using again that $f d\nu_{\rho_*} \prec d\nu_{\rho_0}$,
 in order to prove the Two Blocks Estimate we only need to
show that
\begin{multline}\label{testa}
\limsup _{N \uparrow \infty, \e \downarrow 0, \ell \uparrow
\infty } \sup_{f\in \Upsilon_{C_0,N} }\\
\int \frac{1}{N^d (\e N) ^{ 2d } \ell^d } \sum _{x \in \L_N } \sum
_{ y
\in \L_{x, \e N} } \sum_{\substack{ z \in \L_{x, \e N }:\\
|z-y|_\infty \geq 2\ell } }   \bigl| \cN( \L_{y, \ell} )  - \cN(
\L_{z, \ell} ) \bigr| f(\eta) \nu _{\rho_*} (d \eta ) = 0 \,.
\end{multline}
 Let us now make an observation that
will be frequently used below. Let $X \subset \bbZ^d$ be a subset
possibly depending on $\o$ and on some parameters (for simplicity,
we consider a real--value parameter $L \in \bbN$). Suppose that for
$\bbQ$--a.a. $\o$ it holds
$$\limsup _{N \uparrow \infty,L\uparrow \infty}  \frac{|X \cap \L_N|
}{N^d}=0\,.
$$
Then,  in order to prove \eqref{testa} we only need to show that
\begin{multline}\label{testabisso}
\limsup_{N\uparrow \infty, \e \downarrow 0 , \ell \uparrow \infty ,
L\uparrow \infty }
 \sup_{f\in \Upsilon_{C_0,N} }\\
\int \frac{1}{N^d (\e N) ^{ 2d } \ell^{d} } \sum _{x \in \L_N } \sum
_{ y
\in \L_{x, \e N} } \sum_{\substack{ z \in \L_{x, \e N }:\\
|z-y|_\infty \geq 2\ell } }   \bigl| \cN( \L_{y, \ell}\setminus X) -
\cN(  \L_{z, \ell}\setminus X ) \bigr| f(\eta) \nu _{\rho_*} (d \eta
) = 0 \,.
\end{multline}

We know that there exists $\a_0>0$ such that  for each  $\a\in
(0,\a_0] $ the random field  $\hat \o _\a$ defined in
\eqref{paranza3} is a supercritical Bernoulli bond percolation. Let
us write $\cC_\a (\o) $ for the associated infinite cluster. By
ergodicity,
\begin{equation}\label{testacoda}
\limsup_{ N\uparrow \infty, \a \downarrow 0} |(\cC \setminus
\cC_\a)\cap \L_N|/|\L_N|=\lim_{\a \downarrow 0}
 \bbQ( 0 \in \cC \setminus
\cC_\a)=0\,.
\end{equation}
Hence, due to the above considerations, \eqref{testa} is proven if
we show that, for each $\a\in (0,\a_0]$, it holds \eqref{testa} with
$\cC$ replaced by $\cC_\a$. Moreover, since $f d \nu_{\rho_*}
\prec\nu _{\rho_0}$, using Chebyshev inequality it is simple to
prove that
\begin{multline*}
\limsup_{N\uparrow \infty, \e \downarrow 0, \ell \uparrow \infty, A
\uparrow \infty }\int \frac{1}{N^d (\e N)^{2 d} }
 \sum _{x \in \L_{N} } \sum _{ y
\in \L_{x, \e N} } \sum_{\substack{ z \in \L_{x, \e N }:\\
|z-y|_\infty \geq 2\ell } }  \\ \bbI( \cN( \L_{y, \ell}\cup
\L_{z,\ell} )> A \ell_* ^d) f(\eta) \nu _{\rho_*} (d \eta )=0\,.
\end{multline*}
At this point, we only need to prove the following: Fixed $\a\in (0,
\a_0]$ and $A>0$, for $\bbQ$--a.a. $\o$ it holds
\begin{multline}\label{piedi}
\limsup _{N \uparrow \infty,\e\downarrow 0,  \ell \uparrow \infty}
\sup_{f\in \Upsilon^*_{C_0,N} } \int \frac{1}{N^d (\e N)^{2 d} \ell
^{d} }
 \sum _{x \in \L_{N} } \sum _{ y
\in \L_{x, \e N} } \sum_{\substack{ z \in \L_{x, \e N }:\\
|z-y|_\infty \geq 2\ell } }  \\
 \bigl| \cN( \cC_\a \cap \L_{y,\ell}
) - \cN (\cC_\a \cap \L_{z,\ell} ) \bigr|\bbI ( \cN( \L_{y,
\ell}\cup \L_{z,\ell} )\leq A \ell_* ^d) f(\eta) \nu _{\rho_*} (d
\eta ) = 0 \,,
\end{multline}
where $\Upsilon^*_{C_0,N}$ is  the family of measurable functions
$f:\bbN^{\cC(\o)}\rightarrow [0,\infty)$ such that
$\nu_{\rho_*}(f)=1$ and $\cD (f) \leq C_0 N^{d-2}$.

\bigskip

 We now use the results of \cite{AP} about the chemical distance in the
 supercritical Bernoulli bond percolation  $\hat \o _\a$, for some fixed $\a \in (0, \a_0]$.
  We fix a
positive integer $K$ (this corresponds to the parameter $N$ in
\cite{AP}, which is fixed large enough once for all). Given $\bfa
\in \bbZ^d$ and $s>0$, we set $ \D _{\bfa, s} := \L_{(2K+1) \bfa, s
}$.
 As in \cite{AP}, we call $\hat \o_\a$ the
microscopic random field. The macroscopic one  $\s = \{ \s(\bfa ) :
\bfa \in \bbZ^d\} \in \{0,1\}^{\bbZ^d} $ is defined in \cite{AP}
stating that $\s (\bfa )=1$ if and only if the microscopic field
$\hat \o _\a$ satisfies certain geometric properties inside the box
$ \D_{\bfa, 5K/4}$. These properties are described on page 1038 in
\cite{AP}, but their content is not relevant here, hence we do not
recall them.  What is relevant for us is that there exists a
function $\bar p: \bbN \rightarrow [0,1)$  with $\lim _{K\uparrow
\infty} \bar p (K)=1$, such that $\s$ stochastically dominates a
Bernoulli site percolation of parameter $\bar p (K)$ (see
Proposition 2.1 in \cite{AP}).
  Below we  denote by $\bbP_{\bar p (K) }$ the law of this last
  random field, taking  $K$ large enough such that
  $\bar p (K)$ is supercritical.
 As in \cite{AP} we call a point $\bfa \in \bbZ^d$
white or black if $\s(\bfa)=1$ or $0$ respectively,    and we write
in boldface the sites referred to the macroscopic field.  Recall
that a subset of $\bbZ^d$ is $*$--connected if it is connected with
respect to the adjacency relation
$$ x \stackrel{*}{ \sim }y \Leftrightarrow  |x-y|_\infty=1\,.$$
$\cC^*$ is defined as the set of all $*$--connected macroscopic
black clusters. Given $\bfa \in \bbZ^d$,    $\bfC ^*_{\bfa} $
denotes the element of $\cC^*$ containing $\bfa$ (with the
convention that $\bfC ^*_{\bfa}= \emptyset$ if $\bfa$ is white),
while $\bar{\bfC}^*_\bfa$ is defined as $\bar{\bfC}^*_\bfa=
\bfC^*_\bfa \cup \partial^{out} \bfC^*_\bfa$. We recall that, given
a finite subset $\Lambda \subset \bbZ^d$, its outer boundary is
defined as
$$ \partial^{out} \Lambda:=
 \bigl\{ x\in \L^c\,:\, \exists y \in \L,
\; \{x,y\} \in  \mathbb{E}_d\bigr\}\,.
$$
We use the convention that $\partial ^{out} \bfC ^*_\bfa =\{\bfa\}$
for a white site $\bfa \in \bbZ^d$. Hence, for $\bfa$ white it holds
$\bar{\bfC}^*_\bfa=\{\bfa\}$.

\smallskip

 Let us recall the first part of
Proposition 3.1 in \cite{AP}. To this aim,  given $x, y \in \bbZ^d$,
we write $\bfa (x) $ and $\bfa (y)$ for the unique sites in $\bbZ^d$
such that $x \in \D_{\bfa, K}$ and $y \in \D_{\bfa ,K}$. We set $n:=
|\bfa (x) -\bfa (y)|_1$ and choose a macroscopic path
$\bfA_{x,y}=(\bfa_0,\bfa_1, \dots , \bfa _n)$ with $\bfa_0=\bfa (x)$
and $\bfa _n = \bfa (y)$ (in particular, we require that $|\bfa _i-
\bfa _{i+1}|_\infty =1 $). We build the path $\bfA_{x,y}$ in the
following way: we start in $\bfa (x)$, then  we move by unitary
steps along the line $\bfa (x)+ \bbZ e_1$ until reaching the point
$\bfa'$ having the same first coordinate as $\bfa (y)$, then we
move by unitary steps along the line $\bfa' + \bbZ e_2$ until
reaching the point having the same first two coordinates as $\bfa
(y)$ and so on. Then, Proposition 3.1 in \cite{AP} implies (for $K$
large enough, as we assume) that given any points $x,y \in \cC _\a$
there exists a path $\g_{x,y}$ joining $x$ to $y$ inside $\cC_\a$
such that $\g_{x,y}$ is contained in
\begin{equation}\label{passaggio}
 W_{x,y} := \cup
_{\bfa \in \bfA _{x,y}} \left( \cup _{\bfw \in \bar C ^* _\bfa }
\D_{\bfw , 5K/4 }\right)\,.
\end{equation}
%We recall another  property from \cite{AP} (see formula (4.47))
%there): if $K$ is chosen large enough,  then $\bbE_{\bbQ}( \bfC
%_{\mathbf{0}}  ^*) <\infty$. Since the random variables
%$|\bfC_\bfa^*|$, $\bfa \in \bbZ^d$, are identically distributed,
 %we conclude
%that, if $K$ is chosen large enough (as we assume), then
%\begin{equation}\label{campana}
%\sup _{\bfa \in \bbZ^d} \bbE_{\bbQ}\bigl( | \bfC_\bfa ^*|\bigr)<\infty\,.
%\end{equation}

%\club \club  Finally, we point out that, by the same arguments used
%in the proof of Theorem 1.1 in \cite{AP},  one can prove the
%following: for $\bbQ$--a.a. $\o$ there exists a constant $C>0$ such
%that  the path $\g_{u,v}$ has length bounded by $C \e N $ for all
%$x,u,v$ with $x \in \L_N$, $u,v \in \cC_\a\cap \L_{x,N}$ and for all
%$N \geq 1$.

\smallskip

 These are  the  main results of \cite{AP} that we will use
below. Note that, since the sets $\bar \bfC_\bfa ^*$ can be
arbitrarily large, the information that $\g_{x,y} \subset W_{x,y}$
is not strong enough to allow to repeat the usual arguments in order
to prove the Moving Particle Lemma, and therefore the Two Blocks
Estimate. Hence, one needs some new ideas, that now we present.

\smallskip

 First, we isolate a set of bad points as follows. We fix a
parameter $L>0$ and  we define the subsets $\bfB(L), B(L) \subset
\bbZ^d$ as
\begin{align}
& \bfB(L):=\{ \bfa\in \bbZ^d\,:\, |\bfC^*_\bfa |>L\}\,,\label{gola1} \\
&  B(L):= \cup_{\bfa \in \bfB (L)} \D_{\bfa , 10 K} \label{gola2}\,.
\end{align}
%Note that, \club using \emph{preclusters} as in the proof of Theorem
%1.1. of \cite{A},  one can prove
% for $\bbQ$--a.a. $\o$ that
%  \begin{multline}
% \limsup_{N\uparrow
%\infty} \frac{ | B(L) \cap \L_N|}{|\L_N|} \leq \lim_{N\uparrow
%\infty} c(K) \frac{|\bfB (L) \cap \L_N|}{|\L_N|} = \lim_{N\uparrow
%\infty} c(K) Av_{\bfa \in \L_N} \chi ( |\bfC_\bfa ^*|>L) \leq\\
%c'(K) \bbQ (|\bfC_{\mathbf{0} } ^*|>L)\leq c'(K)\bbE_{\bbQ}
%(|\bfC_{\mathbf{0} } ^*|)/L\,.
%\end{multline}
%Due to \eqref{campana}, we conclude that
\begin{Le}\label{dindondan}
Given $\a$  in  $(0,\a_0]$, for $\bbQ$--a.a. $\o$ it holds
\begin{equation}\label{dindon}
 \limsup_{N\uparrow \infty, L \uparrow \infty} \frac{ | B(L) \cap \L_N|}{|\L_N|} =0\,.
\end{equation}
\end{Le}
\begin{proof}
Since $| B(L) \cap \L_N|  \leq  c(K) |\bfB (L) \cap \L_N|$, we only
need to prove the thesis with $B(L)$ replaced by $\bfB (L)$.
 We
introduce the nondecreasing function $\rho_L:\bbN\rightarrow
[0,\infty)$ defined as $\rho_L(n):= \bbI( n
> L) n$. Then we can bound
$$ | \bfB (L) \cap \L_N| \leq \sum _{ \substack{ \bfC^* \in \cC^*:\\
\bfC^* \cap \L_N \not = \emptyset} } \rho_L(\bfC^*) \,.
$$
Since $\s$ stochastically dominates the Bernoulli site percolation
with law $\bbP_{\bar p (K)}$ and due to Lemma 2.3 in \cite{DP}, we
conclude that
\begin{multline}\label{setina}
\bbQ ( | \bfB (L) \cap \L_N|  \geq a |\L_N|)\leq  \bbQ  \Big (\sum _{ \substack{ \bfC^* \in \cC^*:\\
\bfC \cap \L_N \not = \emptyset} } \rho_L(\bfC^*) \geq a |\L_N|\Big)
\leq \\
\bbP _{\bar p (K) }  \Big (\sum _{ \substack{ \bfC ^*\in \cC^*:\\
\bfC \cap \L_N \not = \emptyset} } \rho_L(\bfC^*) \geq a |\L_N|\Big)
  \leq P\Big ( \sum _{\bfa \in \L_N} \rho_L (\tilde \bfC^* _\bfa )
  \geq a |\L_N|\Big)\,,
\end{multline}
where the random variables $\tilde \bfC^* _\bfa$ (called
\emph{pre--clusters})   are i.i.d. and have the same law of
$\bfC^*_{\mathbf{0} }$ under $\bbP_{\bar p (K)}$. Their construction
is due to Fontes and Newman \cite{FN1}, \cite{FN2}.  Due to formula
(4.47) of \cite{AP}, $E_{\bbP_{\bar p (K) }} (|\bfC^*_{\mathbf{0}
}|)$ is finite for $K$ large, in particular
\begin{equation}\label{claudia} \lim _{L\uparrow \infty} E(\rho_L (\tilde
\bfC^* _\mathbf{0} ))=0\,.\end{equation}
 By applying Cram\'{e}r's theorem, we deduce that
$$P\Big ( \sum _{\bfa \in \L_N} \rho_L (\tilde \bfC^* _\bfa )
  \geq 2  E(\rho_L (\tilde \bfC^* _\mathbf{0} ))   |\L_N|\Big) \leq e ^{ -
  c (L) N^d}\,,
  $$
for some positive constant $c(L)$ and for all $N \geq 1$. Hence, due
to \eqref{setina} and Borel--Cantelli lemma, we can conclude that
for $\bbQ$--a.a. $\o$ it holds
$$  | \bfB (L) \cap \L_N| /
|\L_N|\leq 2E(\rho_L (\tilde \bfC^* _\mathbf{0} ))\,, \qquad \forall
N \geq N_0 (L,\o) \,.
$$
At this point, the thesis follows from \eqref{claudia}.
\end{proof}

At this point, due to the arguments leading to \eqref{testabisso},
we only need to prove the following: given $\a\in(0,\a_0]$ and
$A>0$, for $\bbQ$--a.a. $\o$ it holds
\begin{multline}\label{piedone}
\limsup _{N \uparrow \infty,\e\downarrow 0,  \ell \uparrow \infty,
L\uparrow \infty} \sup_{f\in \Upsilon^*_{C_0,N} } \int \frac{1}{N^d
(\e N)^{2 d} \ell ^{d} }
 \sum _{x \in \L_{N} } \sum _{ y
\in \L_{x, \e N} } \sum_{\substack{ z \in \L_{x, \e N }:\\
|z-y|_\infty \geq 2\ell } }  \\
 \bigl| \cN( \G_{y, \ell, \a} ) - \cN (\G_{z, \ell, \a}  ) \bigr|\bbI( \cN( \G_{y, \ell, \a}\cup \G_{z,\ell,\a}  )\leq A \ell_*^d ) f(\eta) \nu _{\rho_*} (d \eta ) =
0
\end{multline}
where \begin{equation}\label{zaire}
 \G_{u, \ell, \a}= \bigl(\L_{u,
\ell} \cap \cC_\a \bigr) \setminus B(L) \,, \qquad u \in \bbZ^d\,.
\end{equation}
Above we have used also that
$$\bbI( \cN( \L_{y, \ell}\cup \L_{z,\ell}  )\leq A \ell_*^d )\leq  \bbI ( \cN( \G_{y, \ell, \a}\cup \G_{z,\ell,\a}  )\leq A \ell_*^d )\,.
$$
 Note that in the integral of
\eqref{piedone}, the function $f$ multiplies  an $\cF_N$--measurable
function, where $\cF_N$ is the $\s$--algebra generated by the random
variables $\{\eta(x): x \in G_N\}$ and $G_N$ is the set of
\emph{good points} define as
\begin{equation}\label{patty}
 G_N:= (\L_{N+1} \cap \cC_\a )\setminus B(L)\,.
 \end{equation} Since
$\cD(\cdot)$ is a convex functional (see Corollary 10.3 in Appendix
1 of \cite{KL}), it must be $$\cD( \nu_{\rho_*}(f |\cF _N  )) \leq
\cD (f)\leq  C_0 N^{d-2}\,.$$ Hence, by taking the conditional
expectation w.r.t. $\cF_N$ in \eqref{piedone}, we conclude that we
only need to prove \eqref{piedone} by substituting $\Upsilon^*_{C_0,
N}$ with $\Upsilon ^\sharp _{C_0,N}$ defined as the family of $\cF_N
$--measurable functions $f:\bbN^{\cC(\o)} \rightarrow [0,\infty)$
such that
$\nu(f)=1$ and $\cD (f) \leq C_0 N^{d-2}$. %We do not try to prove the
%standard Moving Particle Lemma, but more simply the below lemma. In
%order to state it, we introduce some new notation.

\medskip

Recall the definition of the function $\varphi(\cdot)$ given before
\eqref{pasqua0}. By the change of variable $\eta\rightarrow \eta-
\d_x$ one easily proves the identity
\begin{equation}\label{ghiaccio} \nu_{\rho_*} \left[ g(\eta_x)
\bigl( \sqrt{f}(\eta^{x,y} )- \sqrt{f}(\eta) \bigr)^2\right]=
\varphi(\rho_*)\nu_{\rho_*} \left[ \bigl( \sqrt{f}(\eta^{x,+} )-
\sqrt{f}(\eta^{y,+} ) \bigr)^2\right]\,,
\end{equation}
where in general $\eta^{z,+} $ denotes the configuration obtained
from $\eta$ by adding a particle at site $z$, i.e.\ $\eta^{z,+}=
\eta+ \d_z$.  Let us write $\nabla_{x,y}$ for the  operator
$$ \nabla_{x,y} h (\eta) := h(\eta^{x,+})- h (\eta^{y,+} )\,.$$
We can finally state our weak version of the Moving Particle Lemma:
\begin{Le}\label{media}
For $\bbQ$--a.a. $\o$ the following holds. Fixed  $\a\in (0,\a_0]$
and $L>0$, there exists a positive constant $\k=\k(L,\a)$ such that
\begin{multline}\label{scuola}
\frac{\e^{-2}\varphi(\rho_*) }{N^d (\e N)^{2 d} \ell_* ^{2d} }
 \sum _{x \in \L_{N} } \sum _{ y
\in \L_{x, \e N} } \sum_{\substack{ z \in \L_{x, \e N }:\\
|z-y|_\infty \geq 2\ell } } \sum_{u \in  \G_{y, \ell, \a} }
\sum_{v\in \G_{z, \ell, \a}  } \nu_{\rho_*}  \bigl( (\nabla_{u,v}
\sqrt{f})^2 \bigr) \\ \leq
  N^{2-d}  \cD(f)/\k\leq  C_0 /\k\,,
\end{multline}
for any function $f \in \Upsilon^\sharp_{C_0,N}$ and for any
$N,\ell,C_0$.
\end{Le}
\begin{proof}  Recall the definition of the path $\g_{x,y}$ given for
$x, y \in \cC_\a$ in the discussion before \eqref{passaggio}. Given
a bond $b$ non intersecting $G_N$, since $f $ is $\cF_N$--measurable it holds
$\nabla_b \sqrt{f}=0$. Using this simple observation,
 by a standard telescoping argument together with  Schwarz inequality,
 we obtain that
\begin{equation}
\nu_{\rho_*}
 \bigl( (\nabla_{u,v} \sqrt{f})^2 \bigr)
 \leq \Big\{\sum_{i=0}^{n-1} \bbI( \{u_i, u_{i+1} \} \cap G_N \not = \emptyset)
\Big\}\cdot \Big\{ \sum _{i=0}^{n-1}
  \nu_{\rho_*}\bigl((\nabla_{u_i,u_{i+1}}\sqrt{f})^2\bigr)\Big\}\,,
\end{equation}
where the  path $\g_{u,v}$ is written as  $(u=u_0,u_1,\dots,
u_n=v)$. Recall that if $
\nu_{\rho_*}\bigl((\nabla_{u_i,u_{i+1}}\sqrt{f})^2\bigr)\not =0$
then the set $\{u_i,u_{i+1}\}$ must intersect the set of good points
$ G_N$ defined in \eqref{patty}.

 If $b$ is a bond of $\g_{u,v}$, then $b$ must
be contained in the set $W_{u,v}$ defined in \eqref{passaggio}. In
particular, there exists $\bfa \in \bfA _{u,v}$  and  $\bfw\in \bar
\bfC_\bfa ^*$ such that $b$ is contained in $\D_{\bfw, 5K/4}$.
%\subset \D_{\bfw, 10K}$.
Denoting $d_\infty (\cdot, \cdot) $ the distance between subsets of
$\bbZ^d$ induced by the uniform norm $|\cdot|_\infty$,  we can write
\begin{equation}\label{mmm}
d_\infty( b, (2K+1) \bfw ) \leq 5K/4 \,.
\end{equation}
We claim  that, if $b$ intersects $G_N$, then $|\bfC_\bfa ^*|\leq L$.  If $\bfa $ is
white then $|\bfC_\bfa ^*|=1$ and the claim is trivally true. Let us
suppose that  $\bfa $ is black and that $|\bfC_\bfa ^*|> L$. By
definition of $\bar \bfC_\bfa ^*$, there exists some point $\bfa'
\in \bfC_\bfa ^*$ such that $|\bfw -\bfa '|_\infty \leq 1$, i.e.
\begin{equation}\label{mmm0}
d_\infty\bigl((2K+1) \bfw, (2K+1) \bfa'\bigr)\leq 2K+1\,.
\end{equation}
Due to \eqref{mmm} and \eqref{mmm0} we conclude that $b \subset
\D_{\bfa ', 10K}= \L_{(2K+1)\bfa', 10 K}  $.
 On the other hand,
$\bfC_{\bfa '} ^*=\bfC_\bfa ^*$ and by definition of the set of bad
points we get  that $\D_{\bfa', 10K}\subset B(L)$. The above
observations imply that $b \subset B(L)$ in contradiction with the
fact that $b$ intersects the set of good points $G_N$. This
concludes the proof of our claim: $|\bfC_\bfa ^*|\leq L$.

\smallskip

We define $|\g_{u,v}|_*:=\sum_{i=0}^{n-1} \bbI( \{u_i, u_{i+1} \} \cap G_N \not = \emptyset)$.
 We claim that for almost all conductance field $\o$  there exists a constant $c(K,L)$ depending only on $K$ and $L$
  such that, for
all $(x,y,z,u,v)$ as in \eqref{scuola},   $|\g_{u,v}|_*\leq c \e N$. Indeed, by the above claim,
we get that
$$ | \g_{u,v}|_* \leq  |\L _{5K/4} |  \sum _{\bfa \in \bfA_{u,v} } |\bar C ^*_\bfa
 |  \bbI( | \bfC _\bfa ^*|\leq L)\leq c(d) L |\L _{5K/4} |  \, |\bfA_{u,v}|\leq C(K,L) \e N\,.
$$
\smallskip

 Given a bond $b \in \bbE_d$
let us estimate the cardinality of the set $\bbX(b)$, given by the
strings $(x,y,z,u,v)$ with $x,y,z,u,v$ as in the l.h.s. of
\eqref{scuola}, such that $b$ is a bond of the path $\g_{u,v}$ and
$b$ intersects $G_N$.
Up to now we know that   there
exist $\bfw, \bfa $ such that $b$ intersects $\D_{\bfw, 5K/4}$,
$\bfa \in \bfA _{u,v}$,  $\bfw\in \bar \bfC_\bfa ^*$ and $|\bfC_\bfa
^*|\leq L$. In particular, it must be
\begin{multline}
 d_\infty ( b, (2K+1) \bfA _{u,v} )\leq\\
d_\infty( b, (2K+1) \bfw)+(2K+1) d_\infty(  \bfw, \bfC _\bfa)+ (2
K+1) \text{diam} ( \bfC_\bfa
 ^*)\leq \\ 5K/4+(2K+1)(1+L)\,.
 \end{multline}
Hence, if $(x,y,z,u,v)\in \bbX(b)$, then the distance between $b$
and $(2K+1)\bfA_{u,v}$ is bounded by some constant depending only on
$K$ and $L$. Note that the macroscopic path $\bfA_{u,v}$ has length
bounded by $c \e N /K$.
 Let us consider now the set $\mathbb{Y}(b)$ of macroscopic path $(\bfa _0, \bfa _1,
 \dots, \bfa _k)$ such that
 \begin{align}
 & k \leq c \e N/K\,,\\
 &  d_\infty\left(
 b, (2K+1)\bigl\{\bfa _0, \bfa _1,
 \dots,
 \bfa _k\bigr\}\right)\leq 5K/4+(2K+1)(1+L)\,.
\end{align}
By the same computations used in the proof of the standard Moving
Particle Lemma, one easily obtains that $|\mathbb{Y}(b)|\leq  c(K,L)
(\e N)^{d+1}$. Fixed  a macroscopic path in $\mathbb{Y}(b)$  there
are  at most $c(K) $ ways to choose microscopic points  $u,v$  such
that the macroscopic path equals $\bfA_{u,v}$ . Fixed also $u$ and
$v$, we have at most $ \ell_*^{2d}$ ways to choose $(y,z)$ as in
\eqref{scuola}. Fixed $(y,z)$ we have at most $c(\e N)^d$ ways to
choose $x$ as in \eqref{scuola}. Hence $ |\bbX (b)| \leq c(K,L) (\e
N) ^{2d+1} \ell_*^{2d}$. Finally, recall that the length of
$\g_{u,v}$ is bounded by $c\e N$.
 Since moreover all paths $\gamma_{u,v}$ are in $\cC_\a$, from the above observations and from \eqref{ghiaccio}
  we derive that
\begin{multline}
\text{l.h.s. } of \eqref{scuola} \leq \frac{(c \e N)
\e^{-2}\varphi(\rho_*)}{N^d (\e N)^{2 d} \ell_* ^{2d} }\sum _{b \in
\bbE_d} |\bbX (b) | \nu_{\rho_*}( (\nabla_b
\sqrt{f})^2 ) \leq \\
  \frac{c(K,L) (\e N) ^{2d+2} \e^{-2}\ell_*^{2d}}{ N^d (\e N)^{2 d}
\ell_* ^{2d} \a}
 \cD (f)\leq C(K,L,\a ) N^{2-d}  \cD (f) \leq C(K,L,\a)C_0 \,.
\end{multline}
\end{proof}

Since $\a$ is fixed, we will stress only  the dependence of the
constant $\k(L,\a)$ in  Lemma \ref{media}  on $L$ writing simply
$\k(L)$.
 We introduce  the function
\begin{multline*} \cZ(\eta):=\frac{1}{N^d (\e N)^{2 d}}
 \sum _{x \in \L_{N} } \sum _{ y
\in \L_{x, \e N} } \sum_{\substack{ z \in \L_{x, \e N }:\\
|z-y|_\infty \geq 2\ell } }  \\
 \Big| \frac{\cN( \G_{y, \ell, \a} ) }{\ell_*^d}- \frac{\cN (\G_{z,
\ell, \a} )}{\ell_*^d} \Big|\bbI ( \cN( \G_{y, \ell, \a}\cup
\G_{z,\ell,\a} )\leq A \ell_*^d)\,.
\end{multline*}
 Recall that we need to prove \eqref{piedone} where
$\Upsilon ^* _{C_0,N} $ is substituted by the family
$\Upsilon^\sharp _{C_0,N}$ defined before \eqref{ghiaccio}. Due to
the above Lemma, we can substitute $\Upsilon^\sharp_{C_0,N}$ with
the family $\Upsilon^\natural_{C_0/\k(L),N}$ of measurable functions
$f : \bbN^{\cC(\o) }\rightarrow [0,\infty)$ such that $\nu(f)=1$ and
such that the l.h.s. of \eqref{scuola} is bounded by $C_0/\k(L) $.
Namely, we need to prove that, given $\a \in (0, \a_0]$ and  $A>0$,
for $\bbQ$--a.a. $\o$  it holds
\begin{equation}
\label{prendo} \limsup _{N \uparrow \infty,\e\downarrow 0, \ell
\uparrow \infty, L\uparrow \infty}\sup_{f \in
\Upsilon^\natural_{C_0/\k(L),N}} \nu_{\rho_*}\bigl( \cZ (\eta)
f(\eta) \bigr)=0\,.
\end{equation}
Since by Lemma \ref{media} $$\sup_{f \in
\Upsilon^\natural_{C_0/\k(L),N}} \nu_{\rho_*} (\sqrt{f},-
\frac{\e^{-2} \k (L)}{N^d (\e N)^{2 d} \ell_*^{2d}}
 \sum _{x \in \L_{N} } \sum _{ y
\in \L_{x, \e N} } \sum_{\substack{ z \in \L_{x, \e N }:\\
|z-y|_\infty \geq 2\ell } }  \sum _{u \in \G_{y, \ell, \a} } \sum_{v
\in \G_{z, \ell, \a } } \nabla_{u,v}  \sqrt{f})\leq C\,,
$$
 we only need to prove that, given $\a
\in (0, \a_0]$ and  $A,\gamma
>0$, for $\bbQ$--a.a. $\o$ it holds
\begin{equation}\label{nuvole}
\limsup _{N \uparrow \infty,\epsilon\downarrow 0,  \ell \uparrow
\infty, L\uparrow \infty} \sup spec_{L^2 (\nu_{\rho_*})} \Big \{
%\frac{1}{N^d (\e N)^{2 d}}
% \sum _{x \in \L_{N} } \sum _{ y
%\in \L_{x, \e N} } \sum_{\substack{ z \in \L_{x, \e N }:\\
%|z-y|_\infty \geq 2\ell } }  \\
%\Big[  \Big| \frac{\cN( \G_{y, \ell, \a} ) }{\ell_*^d}- \frac{\cN
%(\G_{z, \ell, \a} )}{\ell_*^d} \Big|\chi ( \cN( \G_{y, \ell, \a})
%\lor \cN (\G_{z,\ell,\a} )\leq A \ell_*^d)
\cZ (\eta) + \frac{\g \k (L)}{ \e^{2} \ell^{2d}} \sum _{u \in \G_{y,
\ell, \a} } \sum_{v \in \G_{z, \ell, \a } } \nabla_{u,v} \Big]
\Big\} \leq 0\,,
\end{equation}
where $\sup spec_{L^2 (\nu_{\rho_*})}  (\cdot)$ denotes the supremum
of the spectrum in $L^2(\nu_{\rho_*})$ of the given operator. Now we
use the subadditivity property
$$\sup spec_{L^2(\nu_{\rho_*})}(\sum_i
X_i) \leq \sum_i \sup spec_{L^2(\nu_{\rho_*})}( X_i)$$ to bound the
l.h.s. of \eqref{nuvole} by
\begin{multline}\label{motori}
\limsup _{N \uparrow \infty,\e\downarrow 0,  \ell \uparrow \infty,
L\uparrow \infty}
 \frac{1}{N^d (\e N)^{2 d}  }
 \sum _{x \in \L_{N} } \sum _{ y
\in \L_{x, \e N} } \sum_{\substack{ z \in \L_{x, \e N }:\\
|z-y|_\infty \geq 2\ell } }  \sup spec_{L^2 (\nu_{\rho_*})} \Big \{
\Big| \frac{\cN( \G_{y, \ell, \a} )}{\ell_*^d} -\frac{ \cN (\G_{z,
\ell, \a} )}{\ell_*^d}
\Big|\cdot \\
\bbI ( \cN( \G_{y, \ell, \a}\cup \G_{z,\ell,\a} )\leq A \ell_*^d ) +
\frac{\g \k(L)}{ \e^{2} \ell_*^{2d}} \sum _{u \in \G_{y, \ell, \a} }
\sum_{v \in \G_{z, \ell, \a } } \nabla_{u,v} \Big\}\,.
\end{multline}
We observe that the operator inside the $\{\cdot\}$--brackets
depends only on $\eta$ restricted to $\G_{y,z}:= \G_{y, \ell, \a}
\cup \G_{z, \ell, \a}$. By calling $\nu_{k, y, z}$ the canonical
measure on $\cS _{k,y,z}:=\{\eta \in \bbN^{\G_{y,z}} : \cN (
\G_{y,z} )=k\}$ obtained by conditioning the marginal of
$\nu_{\rho_*}$ on $\bbN ^{\G_{y,z}}$ to the event
$\{\cN(\G_{y,z})=k\}$, we can bound \eqref{motori} by
\begin{multline}\label{molecolari}
\limsup _{N \uparrow \infty,\e\downarrow 0,  \ell \uparrow \infty,
L\uparrow \infty}
 \frac{1}{N^d (\e N)^{2 d}  }
 \sum _{x \in \L_{N} } \sum _{ y
\in \L_{x, \e N} } \sum_{\substack{ z \in \L_{x, \e N }:\\
|z-y|_\infty \geq 2\ell_* } }  \sup_{k\in\{0,1,\dots , A \ell_*^d
\}}\\ \sup spec_{L^2 (\nu_{k,y,z} )} \Big \{ \Big| \frac{\cN( \G_{y,
\ell, \a} )}{\ell_*^d}  - \frac{ \cN (\G_{z, \ell, \a} )}{\ell_*^d}
\Big| + \frac{\g \k (L)}{ \e^{2} \ell_*^{2d}} \sum _{u \in \G_{y,
\ell, \a} } \sum_{v \in \G_{z, \ell, \a } } \nabla_{u,v} \Big\}\,.
\end{multline}
Given integers  $k, n_1,n_2 \in \bbN$, define for $j=1,2$ the set
$\G_j :=\{1,2, \dots, n_j\}$, with the convention that $\G_j =
\emptyset $ if $n_j=0$. Then define
  the space
$$ \cS_{k,n_1,n_2} :=\{ (\z_1,\z_2 )\in \bbN^{\G_1} \times
\bbN^{\G_2}: \sum _{a_1 \in \G_1} \z_1(a_1) + \sum_{a_2 \in \G_2}
\z_2(a_2) =k\}
$$
and set $\cN(\z_i) = \sum_{a_i \in \G_i} \z_i (a_i)$.
 Finally, call $\nu_{k, n_1,n_2}$ the probability measure on
$\cS_{k,n_1,n_2}$ obtained by first taking the product measure on
$\bbN^{\G_1 }\times \bbN^{\G_2}$ with the same marginals as
$\nu_{\rho_*}$, and afterwards by conditioning this product measure
to the event that the total number of particles is $k$. Finally,
define
\begin{equation}
F(k,n_1,n_2,\e,\ell,L ):= \sup spec_{L^2 (\nu_{k, n_1,n_2} )} \Big
\{ \ell_*^{-d} \bigl|  \cN (\z_1)- \cN (\z_2)  \bigr| +\frac{ \g \k
(L) }{ \e^{2}\ell_*^{2d}} \sum _{u \in \G_1 } \sum_{v \in \G_2 }
\nabla_{u,v} \Big\}\,.
 \end{equation}
Note that the operator $ \g \k (L)  \ell_*^{-2d} \sum _{u \in \G_1 }
\sum_{v \in \G_2 } \nabla_{u,v}$ is the Markov generator of a
process on $\cS_{k,n_1,n_2}$ such that the measure $\nu_{k,n_1,n_2}$
is reversible and ergodic. In particular, $0$ is a simple
eigenvalue for this process. % We call $\delta(k,n_1,n_2, \ell)$ its spectral gap.
Fixed $\ell$, we will vary the triple $(k,n_1,n_2)$ in a finite set,
more precisely we will take $n_1,n_2\leq \ell_*^d$ and $0\leq k \leq
 A \ell_*^d $. Then,  applying  Perturbation Theory (see  Corollary 1.2 in Appendix 3.1 of
\cite{KL}), we conclude that
\begin{equation}\label{emozioni}
 \limsup_{\e\downarrow 0} \sup_{k,n_1,n_2}\Big |
 F(k,n_1,n_2, \e,\ell,L)- G(k,\ell,n_1,n_2)\Big | =0
\end{equation}
where
\begin{equation}
G(k,\ell,n_1,n_2)= \nu _{k,n_1,n_2} \bigl(\ell_*^{-d} \bigl|  \cN
(\z_1)- \cN (\z_2)  \bigr| \bigr)\,.
\end{equation}
The above result implies that in order to prove that
\eqref{molecolari} is nonnegative we only need to show that
 \begin{multline}\label{sole}
\limsup _{N \uparrow \infty,\e\downarrow 0,  \ell \uparrow \infty,
L\uparrow \infty}
 \frac{1}{N^d (\e N)^{2 d}  }
 \sum _{x \in \L_{N} }
\sum _{ y\in \L_{x, \e N} } \sum_{z \in \L_{x, \e N} }
\sup_{k\in\{0,1,\dots ,  A\ell_*^d \}} G (k,\ell,n_y,n_z)\leq 0\,,
\end{multline}
where $$ n_y=  |\G_{y,\ell, \a }|\,, \qquad n_z= |\G_{z,
\ell,\a}|\,. $$

\begin{Le} Given $\d
>0$, for $\bbQ$--a.a. $\o$ it holds
\begin{multline}\label{universo}
\limsup _{N \uparrow \infty,\e\downarrow 0,  \ell \uparrow \infty, L
\uparrow \infty }
 \frac{1}{N^d (\e N)^{2 d}  }
 \sum _{x \in \L_{N} }
\sum _{ y\in \L_{x, \e N} } \sum_{z \in \L_{x, \e N} }\\ \bbI \bigl(
|n_y /\ell_*^d  - m_\a|>\d \text{ or }  |n_z /\ell_*^d  - m_\a|>\d
\bigr)=0\,,
\end{multline}
where $m_\a= \bbQ ( 0\in \cC_\a)$. \end{Le} \begin{proof} Recall
definition \eqref{zaire} and  set $N_y=|\cC_\a\cap \L_{y,\ell}|$,
$N_z=|\cC_\a\cap \L_{z,\ell}|$. Then \begin{multline*} \frac{1}{N^d
(\e N)^{2 d}  }
 \sum _{x \in \L_{N} }
\sum _{ y\in \L_{x, \e N} } \sum_{z \in \L_{x, \e N} } \bbI \bigl(
|n_y- N_y | >\d \ell_*^d \bigr)\leq\\\frac{c }{N^d (\e N)^{d}  }
 \sum _{x \in \L_{N} }
\sum _{ y\in \L_{x, \e N} }  \bbI(|B(L) \cap \L_{y,\ell} |> \d
\ell_*^d)\leq\\\frac{c }{\d N^d (\e N)^{d} \ell_*^d  }
 \sum _{x \in \L_{N} }
\sum _{ y\in \L_{x, \e N} } |B(L) \cap \L_{y,\ell} |\leq \frac{C
}{\d N^d}  |B(L) \cap \L_{N+\e N} |\,.
\end{multline*}
Due to the above bound,  the same  bound for  $n_z$ and $N_z$, Lemma
\ref{dindondan} and finally a $\d/2$--argument,  we conclude that it
is enough to prove \eqref{universo} substituting  $n_y$ and $n_z$
with $N_y$ and $N_z$ respectively. Moreover, due to ergodicity, for
$\bbQ$--a.a. $\o$ it holds
 \begin{multline}
\limsup _{N \uparrow \infty,\e\downarrow 0,  \ell \uparrow \infty}
 \frac{1}{N^d (\e N)^d }
 \sum _{x \in \L_{N} }
\sum _{ y\in \L_{x, \e N} }
\bbI \bigl( |N_y    /\ell_*^d  - m_\a|>\d\bigr) \\
\leq c
\limsup _{N \uparrow \infty ,   \ell \uparrow \infty}
 \frac{1}{N^d  }
 \sum _{y \in \L_{2N} }
\bbI \bigl( |N_y    /\ell_*^d  - m_\a|>\d\bigr)= c\limsup _{   \ell
\uparrow \infty} \bbQ \bigl( |N_y    /\ell_*^d  - m_\a
|>\d\bigr)=0\,.
\end{multline}
A similar bound can be obtained for $z$ instead of $y$, thus
implying \eqref{universo}.
\end{proof}
Since $G(k,\ell, n_1,n_2) \leq A$ and since
 \begin{multline}\label{suono} G(k,\ell, n_1,n_2) \leq  \nu_{k
,n_1 ,n_2 }\Big( \Big| \frac{\cN(\z_1)}{\ell_* ^d} - \frac{
k}{n_1+n_2  } \frac{n_1}{\ell_*^d }\Big|\Big)+\\  \nu_{k ,n_1 ,n_2
}\Big( \Big| \frac{\cN(\z_2)}{\ell_* ^d} - \frac{ k}{n_1+n_2  }
\frac{n_2}{\ell_*^d }\Big|\Big)+ A  \frac{ |n_1-n_2|}{n_1+n_2}\,,
\end{multline} by the above Lemma in order to prove \eqref{sole}  we
only need to prove for $j=1,2$ that
 \begin{equation}\label{soleaia}
\limsup _{  \ell \uparrow \infty,  \d \downarrow 0}
\sup_{(k,n_1,n_2)\in \cI }
 \nu_{k ,n_1
,n_2 }\Big( \Big| \frac{\cN(\z_j)}{\ell_* ^d} - \frac{ k}{n_y+n_2  }
\frac{n_j}{\ell_*^d }\Big|\Big) = 0\,,
\end{equation}
where \begin{multline} \cI=\Big\{ (k,n_1, n_2 ) \in \bbN ^3:\\
\frac{k}{\ell_*^d}\leq A\,, \; \frac{n_1}{\ell_*^d} \in [m_\a-\d,
(m_\a+\d) \wedge 1]\,,\; \frac{n_2}{\ell_*^d} \in [m_\a-\d,
(m_\a+\d)\wedge 1] \Big\}\,.\end{multline} At this point
\eqref{sole} derives from the local central limit theorem as in
 in \cite{KL}[page 89, Step 6].

\section{Proof of the One Block Estimate}\label{F1}

We use here several arguments developed in the previous section. In
order to avoid repetitions, we will only sketch the proof. As
before, for simplicity of notation we take $M=1$.

Let us define $m_\a:= \bbQ (0 \in \cC_\a)$. Note that $m=\lim_{\a
\downarrow 0 }m_\a$. Setting $ a=\eta^\ell (y)$, $b= \cN ( \cC_\a
\cap \L_{y, \ell})/\ell_*^d$, since $\phi$ is Lipschitz  we can
bound
\begin{multline*}
\Big| m \phi(\frac{a}{m})-m_\a \phi( \frac{b}{m_\a} ) \Big| \leq
|m-m_\a|\phi(\frac{a}{m}) +m_\a|\phi(\frac{a}{m} )-\phi(\frac{b}{m}
) \bigr|+ m_\a|\phi(\frac{b}{m} )-\phi(\frac{b}{m_\a} ) \bigr|\leq
\\
\frac{|m-m_\a|}{m} g^* a + g^* \frac{m_\a}{m}|a-b|+ g^* b
m_\a\bigl|\frac{1}{m}-\frac{1}{m_\a} \bigr|\leq \\
\frac{|m-m_\a|}{m} g^* \eta^\ell (y) + g^* \frac{m_\a}{m \ell_*^d
}\cN(\L_{y,\ell} \setminus \cC_\a) + g^*\cN ( \cC_\a \cap \L_{y,
\ell})\ell_*^{-d} m_\a\bigl|\frac{1}{m}-\frac{1}{m_\a}
\bigr|=:\cG(\eta,\o )\,.
\end{multline*}
 Note that $\cG(\eta)$  is increasing in $\eta$. Hence,  using that
 $f(\eta)
\nu_{\rho_*}(d \eta) \prec f(\eta) \nu_{\rho_0} (d \eta)$, we easily
obtain that
\begin{equation}\label{1blockbla}
\limsup _{N\uparrow \infty, \ell \uparrow \infty, \a \downarrow 0}
\sup_{f\in \Upsilon_{C_0,N} } \int \frac{1}{N^d} \sum _{x \in \L_{N}
} \cG (\t_x \eta, \t_x \o ) f(\eta) \nu _{\rho_*} (d \eta ) = 0 \,.
\end{equation}
Due to the above result and
 reasoning as in the derivation of
\eqref{piedone} where $\Upsilon ^*_{C_0,N}$ can be replaced by
 $\Upsilon ^\sharp_{C_0,N}$ (see the discussion after \eqref{piedone}),
 we only need to prove that given
$C_0>0$, $A>0 $ and $\a\in (0,\a_0]$, for $\bbQ$--a.a. $\o$ it holds
\begin{multline}\label{signorino}
 \limsup_{N\uparrow \infty, \ell\uparrow\infty, L\uparrow \infty}
\sup_{f\in \Upsilon^\sharp _{C_0,N} }  \\
\int Av_{x \in \L_N}  \Big| \ell_*^{-d} \sum _{y \in \G_{x,\ell, \a}
} g (\eta (y)) - m_\a \phi  \bigl(
 \cN ( \G_{x,\ell, \a} )/m_\a  \ell_*^d \bigr)\Big| \bbI\bigl(  \cN ( \G_{x,\ell, \a} )\leq A \ell_*^d\bigr) f(\eta)\nu_{\rho_*} (\eta)  =0\,.
\end{multline}
At this point by the same arguments of Lemma \ref{media}, one can
prove that
\begin{Le}\label{sinfonia}
For $\bbQ$--a.a. $\o$ the following holds. Fixed $\a\in (0,\a_0]$
and $L>0$,  there exists a positive constant $\k=\k(\a,L)$ such that
\begin{multline}\label{scuolabus}
\varphi(\rho_*) N^2 \ell_*^{-2d-2} Av _{x \in \L_{N} }  \sum_{u \in
\G_{x, \ell, \a} } \sum_{v\in \G_{x, \ell, \a}  } \nu_{\rho_*}
\bigl( (\nabla_{u,v} \sqrt{f})^2 \bigr)   \leq
  N^{-d+2}  \cD(f)/\k \leq C_0/\k  \,,
\end{multline}
for any function $f \in \Upsilon^\sharp_{C_0,N}$ and for any
$N,\ell,C_0$.
\end{Le}
Due to the above lemma, as in the derivation of \eqref{nuvole} we
only need to prove that given $\a \in (0, \a_0]$, $A, \g>0$, for
$\bbQ$--a.a. $\o$ it holds:
\begin{multline}\label{nuvolebis}
\limsup _{N \uparrow \infty,  \ell \uparrow \infty, L\uparrow
\infty}
\sup spec_{L^2 (\nu_{\rho_*})} \\
\Big\{
 Av_{x \in \L_N}  \Big|
\ell_*^{-d} \sum _{y \in \G_{x,\ell, \a} } g (\eta (y)) -m_\a \phi
\bigl(
 \cN ( \G_{x,\ell, \a} )/m_\a  \ell_*^d  \bigr)\Big| \bbI\bigl(  \cN ( \G_{x,\ell, \a} )\leq A \ell_*^d\bigr)
 \\ + \frac{\g \k (L) N ^2}{ \ell^{2d+2}} \sum _{u
\in \G_{x, \ell, \a} } \sum_{v \in \G_{x, \ell, \a } } \nabla_{u,v}
\Big] \Big\} \leq 0\,.
\end{multline}
Using subadditivity as in the derivation of \eqref{motori} we can
bound the above l.h.s. by
\begin{multline}\label{motoribis}
\limsup _{N \uparrow \infty, \ell \uparrow \infty, L\uparrow \infty}
Av_{x \in \L_N}
\sup spec_{L^2 (\nu_{\rho_*})} \\
\Big\{
 \Big|
\ell_*^{-d} \sum _{y \in \G_{x,\ell, \a} } g (\eta (y)) - m_\a \phi
\bigl(
 \cN ( \G_{x,\ell, \a} )/m_\a \ell_*^d \bigr)\Big| \bbI \bigl(  \cN ( \G_{x,\ell, \a} )\leq A \ell_*^d\bigr)
 \\ +\frac{\g \k(L)  N ^2 }{\ell^{2d+2}} \sum _{u
\in \G_{x, \ell, \a} } \sum_{v \in \G_{x, \ell, \a } } \nabla_{u,v}
\Big] \Big\} \,.
\end{multline}
Again by conditioning on the number of particles in $\G_{x,\ell,\a}$
and afterwards applying perturbation theory (see \eqref{molecolari},
\eqref{emozioni} and \eqref{sole}), one only needs to show that
\begin{equation}\label{supremo}
 \limsup _{N \uparrow \infty, \ell \uparrow \infty, L\uparrow \infty}
Av_{x \in \L_N}  \sup_{k\in\{0,1,\dots, A\ell_*^d\} }  G(k,\ell,
n_x) =0\,,
\end{equation}
where $n_x:= |\G_{x,\a, \ell}|$,  $\nu_{k, n}$ is the measure on $\{
\z \in \bbN^n: \sum_{i=1}^n \z(i) =k\}$ obtained by taking the
product measure with the same marginals as $\nu_{\rho_*}$ and then
conditioning on the event that the total  number of particles
equals $k$, and where
$$ G(k,\ell, n):= \nu_{k,n}\Big[
 \bigl|
\ell_*^{-d} \sum _{i=1  }^n g (\z(i)) -  m_\a \phi ( k/m_\a \ell_*^d
)\bigr|\Big]\,.
$$
One can prove that for $\bbQ$--a.a. $\o$ it holds
\begin{equation}\label{biriccone}\limsup_{N\uparrow \infty, \ell \uparrow \infty, L \uparrow
\infty} Av_{x \in \L_N} \bbI ( |n_x/\ell_*^d - m_\a|>\d)=0\,,
\end{equation}
for each positive constant $\d>0$. As in \cite{KL}[Chapter 5], one
has that
\begin{equation}\label{davidone}
\lim _{\ell \uparrow \infty} \sup_{ (k,n) \in \cJ }\nu_{k,n}\Big[
 \bigl| n^{-1}  \sum _{i=1  }^n g (\z(i)) -   \phi ( k/n
 )\bigr|\Big]=0\,.\end{equation}
 where $$ \cJ:=\Big \{ (k,n) \in \bbN^2\,:\,\frac{k}{\ell_*^d}  \leq A
 \,, \;\frac{n}{\ell_*^d} \in ( m_\a-\d,
(m_\a+\d)\wedge 1)\Big\}\,. $$
 At this point \eqref{supremo} follows
from \eqref{biriccone} and \eqref{davidone}.

\bigskip

\bigskip

\appendix

\section{Zero range process on $\bbZ^d$ with random conductances }\label{cileno}

 Recall that the enviroment $\o= \bigl( \o(b) \,:\, b \in \bbE_d
 \bigr)$ is given by a family of i.i.d. random variables
 parameterized by the set $\bbE_d$ of non--orientied bonds in
 $\bbZ^d$, $d \geq 2$. We denote by $\bbQ$  the law of $\o$,  we assume
 that $\bbQ( \o (b) \in [0,c_0])=1$  and that $\bbQ ( \o (b)>0)$ is supercritical. We fix a function
$g: \bbN \to [0,\infty)$
 as in Subsection \ref{lavinia_giorgia}. Given a realization of $\o$, we consider the
 zero range process $\eta(t)$ on $\bbZ^d$ whose
  Markov
generator $N^2 \cL $ acts on local functions as
\begin{equation}\label{generatorebimbo}
N^2 \cL  f (\eta) =N^2 \sum _{e\in \cB} \sum _{x \in \bbZ^d}
g(\eta(x)) \o(x,x+e)
 \left( f(\eta ^{x,x+e})- f(\eta)\right)\,,
\end{equation}
where $\cB=\{ \pm e_1, \pm e_2, \dots, \pm e_d\}$, $e_1, \dots, e_d$
being the  canonical basis of $\bbZ^d$. Given an {\sl admissible}
initial distribution $\bar \mu^N$ on $\{0,1\} ^{\bbZ^d}$ (i.e. such
that the corresponding zero range process is well defined), we
denote by $\bbP_{\o, \bar \mu^N}$ the law of $(\eta_t\,:\, t \geq 0
)$. Trivially, the zero range process behaves independently on the
different clusters of the conductance field.

If $\bbQ( \o(b)>0)=1$, then $\bbQ$--a.s. the infinite cluster $\cC
(\o)$ coincides with $\bbZ^d$ and the hydrodynamic behavior of the
zero range process  on $\bbZ^d$ is described by Theorem
\ref{annibale}. If $\bbQ(\o (b)>0)<1$ the bulk behavior of the zero
range process on $\bbZ^d$ is different  due to the presence of
finite clusters acting as traps, as we now explain.

First,  we observe  that the finite clusters cannot be too big.
Indeed, as byproduct of Borel--Cantelli Lemma and Theorems (8.18)
and (8.21) in \cite{G},   there exists a positive constant $\g>0$
such that for $\bbQ$--a.a. $\o$ the following property (P1) holds:

\smallskip

\noindent (P1) for each  $N \geq 1$ and each finite cluster $C$
intersecting the box $[-N, N ]^d$, the diameter of $C$ is bounded by
$\g \ln (1+N)$.

\smallskip

\begin{Le}\label{blocco} Suppose that $\o$ has a unique infinite cluster
$\cC(\o)$ and that $\o$ satisfies the above property (P1).
 Let $G \in C_c (\bbR^d)$ and $\eta \in \bbN ^{\bbZ^d}$ be
such that   $ \eta (x)=0$ for all $x \in \cC(\o)$. Call $\D_G$ the
support of $G$ and call $\bar \D_G$ the set of points $z \in \bbR^d$
having distance from $\D_G$ at most $1$.
 Then there exist positive  constants $N_0(G,\g)$, $C(G,\g)$  depending only on  $G$ and $\g$  such that the zero range process on $\bbZ^d$
with initial configuration $\eta$ satisfies a.s.  the following
properties:  $ \eta_t (x)=0$ for all $x \in \cC(\o)$ and, for $N
\geq N_0$,
\begin{equation}\label{aceace}
 \Big| N^{-d}\sum _{x \in \bbZ^d} G (x/N) \eta(x)-N^{-d}\sum _{x
\in \bbZ^d} G (x/N) \eta_t(x)\Big|\leq C(G,\g) \frac{ \ln
(1+N)}{N^{d+1} } \sum _{x \in \bbZ^d: x/N \in \bar{\D}_G} \eta(x)\,.
\end{equation}
\end{Le}
We point out that the zero range process is well defined when
starting in $\eta$, indeed the dynamics reduces to a family of
independent zero range processes on  the finite clusters, while the
infinite cluster $\cC(\o)$ remains empty.
\begin{proof}
The fact that $\eta_t(x) =0$ with $x \not \in \cC(\o)$ is trivial.
Let us prove \eqref{aceace}. Without loss of generality we can
suppose that $G$ has support in $[- 1, 1]^d$ (the general case is
treated similarly). Let us write $C^N_1,C^N_2, \dots , C^N_{k_N}$
for the family of finite clusters intersecting the box $[-N, N]^d$.
For each cluster $C^N_i$ we fix a point $x^N_i \in C^N_i$. Since by
the property (P1) each $C^N_i$ has diameter at most $\g \ln (1+N)$,
we have
$$ | G(x/N)- G( x^N_i /N) | \leq C(G) \g \ln (1+N) /N\,, \qquad \forall
i: 1\leq i \leq k_N\,, \; \forall x \in C^N_i\,.
$$
The above estimate implies  that
\begin{multline*}
\Big| N^{-d}\sum _{x \in \bbZ^d} G (x/N) \eta_t(x)- N^{-d}\sum
_{i=1}^{k_N} G(x^N_i/N) \sum _{x \in C^N_i  } \eta_t(x) \Big| \leq\\
 C(G) \g N^{-d-1} \ln (1+N)\sum _{i=1}^{k_N} \sum _{x
\in C^N_i  } \eta_t(x)\,.
\end{multline*}
Using now that the number of particles in each cluster is
time--independent and  that  for $N$ large enough $C^N_i \subset
[-2N,2N]$ for all $i =1, \dots, k_N$, we get the thesis.
\end{proof}
As a consequence of  Lemma \ref{blocco}, if $\o$ has a unique
 infinite cluster and if $\o$ satisfies property (P1), then  for
 any admissible initial configuration $\eta_0$ (i.e. such that the zero range process
 on $\bbZ^d$ is well defined when starting in $\eta_0$) and any $G \in C_c (\bbR^d)$,  it holds
 \begin{equation}\label{solareace}
 N^{-d}\sum _{x \in \bbZ^d} G (x/N) \eta_t(x)=
N^{-d} \sum _{x
 \in \cC(\o)} G (x/N) \eta_t(x)+
  N^{-d} \sum _{x \not
 \in \cC(\o)} G (x/N) \eta_0(x)+ o (1)\,,
 \end{equation}
 where  $|o(1)|\leq  C(G,\g)\frac{ \ln
(1+N)}{N^{d+1} } \sum _{x \in \bbZ^d: x/N \in \bar{\D}_G} \eta(x)$.
At this point, denoting by $\bar \mu ^N$ the initial distribution of
 the zero range  process $\eta_t$ on $\bbZ^d$, one can derive the
 hydrodynamic limit of $\eta_t$ if
  the marginals  of $\bar \mu^N$  on $\cC(\o)$ and $\bbZ^d \setminus
 \cC(\o)$, respectively, are associated to  suitable macroscopic
 profiles. In what follows, we discuss a special case where this last property is
 satisfied.

\bigskip

We fix a smooth, bounded nonnegative function $\rho_0 :\bbR^d \to
[0,\infty)$ and for each $N$ we define $\bar \mu ^N$ as the product
probability measure on $\bbN ^{\bbZ^d}$ such that for all $x \in
\bbZ^d$ it holds
\begin{equation}\label{insegno} \bar \mu ^N \bigl( \eta(x)=k \bigr)= \nu _{\rho_0
(x/N) } ( \eta ( x)=k )\,, \end{equation} where $\nu_\rho$ is
defined as in Subsection \ref{lavinia_giorgia} with the difference
that  now it is referred to all $\bbZ^d$ and not only to $\cC(\o)$.

\smallskip

We call  $\mu^N$ the marginal of $\bar{\mu} ^N$ on $\cC(\o)$:
$\mu^N$ is a product probability measure on $\bbN^{\cC (\o) }$
satisfying \eqref{insegno} for all $x \in \cC (\o)$ (note that
$\mu^N$ depends on $\o$). Similarly, we call $\nu _{\rho, \cC(\o) }$
the marginal of $\nu_\rho$ on $\cC(\o) $. By the discussion at the
end of Section \ref{natalino}, if  the smooth profile $\rho_0$
converges sufficiently fast at infinity to a positive constant
$\rho_*$, it holds
\begin{equation}\label{cardio}\limsup _{N\uparrow \infty} N^{-d} H( \mu^{N}|
\nu_{\rho_*, \cC(\o) } ) < \infty\, \qquad \bbQ\text{--a.s.}
\end{equation}

\begin{Th}
Suppose that the bounded smooth profile $\rho_0 : \bbR^d \to
[0,\infty)$ satisfies \eqref{cardio}. Then for all $t>0$, $G \in C_c
(\bbR^d)$ and $\d>0$, for $\bbQ$--a.a. $\o$ it holds
\begin{equation}\label{cammarota1}
\lim _{N\uparrow \infty  } \bbP _{\o, \bar{\mu}^N}  \Big ( \Big|
N^{-d} \sum _{x\in \bbZ^d } G (x /N)\,  \eta  _{ t} (x) - \int
_{\bbR ^d} G (x) \rho (x,t ) dx  \Big|>\d \Big )=0
\end{equation}
where, setting $m = \bbQ ( 0 \in \cC(\o) )$,
\begin{equation}
\rho (x,t)=m  \tilde \rho (x,t)+(1-m) \rho_0 (x)
\end{equation}
and  $\tilde \rho: \bbR^d \times [0,\infty) \rightarrow \bbR$ is the
 unique weak solution of the heat equation
\begin{equation}\label{pasqua3ace}
\partial  _t \tilde \rho =  \nabla \cdot ( \cD \nabla  \phi(\tilde \rho))
\end{equation}
with boundary condition $\tilde\rho_0= \rho $ at $t=0$.
\end{Th}
\begin{proof}
Since $\bar \mu^N$ is a product measure and the dynamics on
different clusters is independent, the process restricted to
$\cC(\o)$ has law $\bbP _{\o, \mu^N}$.
 As discussed at the end of
Section \ref{natalino}, to this last process we can apply Theorem
\ref{annibale}. Since, for $\bbQ$-a.a. $\o$ the initial
distributions  $\mu^N$ are associated to the macroscopic profile $m
\rho_0$, we conclude that for $\bbQ$-a.a. $\o$ it holds
\begin{equation}\label{carota1}
\lim _{N\uparrow \infty  } \bbP _{\o, \bar{\mu}^N}  \Big ( \Big|
N^{-d} \sum _{x\in \cC(\o) } G (x /N) \, \eta  _{ t} (x) -m \int
_{\bbR ^d} G (x) \tilde \rho (x,t ) dx  \Big|>\d \Big )=0\,.
\end{equation}
Let us now consider the evolution outside the infinite cluster. Let
us write
$$
N^{-d} \sum _{x\not \in \cC(\o) } G (x /N) \, \eta _{ 0} (x)= N^{-d}
\sum _{x\in \bbZ^d } G (x /N) \, \eta _{ 0} (x)-N^{-d} \sum _{x \in
\cC(\o) } G (x /N) \, \eta _{ 0} (x)\,.$$  We know that, when $\eta$
is sampled with distribution $\bar \mu^N$,  the addenda in the
r.h.s. converge in probability to $\int G(x) \rho_0(x)$ and $m \int
G(x)\rho_0(dx)$, for $\bbQ$--a.a. $\o$. As a consequence the l.h.s.
converges in probability to $(1-m)\int G(x) \rho_0(x)$ for
$\bbQ$--a.a. $\o$. In addition,  $$\sup_{N\geq 1}\Big\{ \int \bar
\mu^N(d\eta) N^{-d} \sum _{x: x/N\in \bar \D_G }  \eta  (x) \Big\}<
\infty\,.
$$
The above observations and  Lemma \ref{blocco} (cf.
\eqref{solareace}) imply that
\begin{equation}\label{carota2}
\lim _{N\uparrow \infty  } \bbP _{\o, \bar{\mu}^N}  \Big ( \Big|
N^{-d} \sum _{x\not \in \cC(\o) } G (x /N) \, \eta  _{ t} (x) -
(1-m) \int _{\bbR ^d} G (x)  \rho_0 (x ) dx  \Big|>\d \Big )=0\,.
\end{equation}
The thesis then follows from \eqref{carota1} and \eqref{carota2}.
\end{proof}

\bigskip

\bigskip

\noindent {\bf Acknowledgements}. The author kindly acknowledges the
Department of Mathematics of the University of L'Aquila for the kind
hospitality while part of this work was being done. She also thanks
the anonymous referee for suggesting the problem discussed in
Appendix \ref{cileno}.

\end{document}